\documentclass{article}

\usepackage[utf8]{inputenc}
\usepackage{graphicx}
\usepackage[normalem]{ulem}
\usepackage{mathrsfs}
\usepackage{amssymb}
\usepackage{amsfonts, mathtools}
\usepackage{amssymb,amsmath,amsthm,tikz-cd, tikz}
\usepackage{xcolor}
\usepackage{hyperref}
\hypersetup{
    colorlinks,
    linkcolor={blue!80!black},
    citecolor={blue!80!black},
    urlcolor={blue!80!black}
}
\usepackage{caption}
\usepackage{subcaption}
\usetikzlibrary{shapes.geometric}
\usepackage{wasysym}
\usepackage{pifont}
\usetikzlibrary{calc}
\usetikzlibrary{decorations.markings}

\usepackage{epigraph}

\usepackage[margin=1in,marginparwidth=0.8in, marginparsep=0.1in]{geometry}

\usepackage{tikz-3dplot}

\tdplotsetmaincoords{60}{120} 

\title{Fukaya category of symmetric product from gluing}
\author{...}

\def\R{\mathbb{R}}
\def\C{\mathbb{C}}
\def\Z{\mathbb{Z}}
\def\N{\mathbb{N}}

\def\In{\subset}

\def\d{\partial}

\def\wt{\widetilde}
\def\wb{\overline}

\def\RM{\backslash}
\def\into{\hookrightarrow}

\DeclareMathOperator{\Sect}{Sect}

\DeclareMathOperator{\Sym}{Sym}
\DeclareMathOperator{\Clus}{Clus}

\DeclareMathOperator{\Fuk}{Fuk}

\DeclareMathOperator{\colim}{colim}
\DeclareMathOperator{\dist}{dist}
\DeclareMathOperator{\sep}{sep}

\newcommand{\ootimes}{
\ooalign{$\otimes$\cr\hidewidth$\bigcirc$\hidewidth}}

\renewcommand{\Im}{\mathrm{Im}}
\renewcommand{\Re}{\mathrm{Re}}

\def\bz{\mathbf{z}} 
\def\bx{\mathbf{x}} 
\def\ba{\mathbf{a}} 

\usepackage{color}

\newtheorem{theorem}{Theorem}[section]
\newtheorem{lemma}[theorem]{Lemma}
\newtheorem{corollary}[theorem]{Corollary}
\newtheorem{proposition}[theorem]{Proposition}

\theoremstyle{definition}

\newtheorem{definition}
[theorem]{Definition}

\newtheorem{example}[theorem]{Example}
\newtheorem{remark}[theorem]{Remark}

\newcommand{\triplerightarrow}{%
  \mathrel{\raisebox{.3ex}{$\rightarrow$}%
  \kern-0.6em\rightarrow\kern-0.6em%
  \raisebox{-.3ex}{$\rightarrow$}}}

\title{Topological algebra of symplectic geometry of symmetric powers}
\author{Vivek Shende and Peng Zhou}
\date{}

\begin{document}

\maketitle

\begin{abstract}
    To a noncompact orientable surface with no closed boundary, we associate the sum of Fukaya categories of (Liouville sectors associated to) its symmetric powers.  We construct sectorial covers with the combinatorics of the bar resolution to show this association extends to an open 2d topological field theory -- without naming a Lagrangian, let alone  a holomorphic disk.  
    In particular, we recover results of Rouquier and Manion on extending  Heegaard-Floer theory down to an interval.  
\end{abstract}


\section{Introduction}

The `Heegaard-Floer' homology of Oszvath and Szabo, now a fundamental theoretical and computational tool in low dimensional topology, is by definition constructed from the Floer homology of certain Lagrangians in symmetric products of surfaces \cite{Ozsvath-Szabo-3manifolds, Ozsvath-Szabo-knots}.  The theory was motivated by four dimensional gauge theory -- and indeed,  provides four dimensional invariants \cite{Ozsvath-Szabo-four-manifolds} -- but is also  luxuriously computable \cite{sarkar-wang, manolescu-ozsvath-sarkar}.  Such computations have allowed at least the simplest `hat' variant to be extended down to categories associated to two  dimensional punctured surfaces \cite{Lipschitz-Ozsvath-Thurston}, which in turn have been shown to have some gluing properties along one dimensional intervals \cite{Douglas-Manolescu, Douglas-Lipschitz-Manolescu, rouquier-manion}.  

Here we give a purely geometric argument that the categories associated to surfaces form a 2d open topological field theory. This result relies on no past or present holomorphic disk computations, is fully coherent, and independent of  coefficients.  The following category organizes our discussion: 

\begin{definition}
    Let $\mathscr{O}$ be the topological category of oriented, possibly disconnected, finite type,  {\em noncompact} surfaces-with-boundary, all of whose boundary components are noncompact (i.e.  $\R$ rather than $S^1$). Morphisms are proper inclusions for which the preimage of a nonempty boundary component is a union of nonempty boundary components.   We give this category the symmetric monoidal structure of disjoint union.
 
\end{definition}

Up to contractible choices, we could equivalently define the objects of $\mathscr{O}$ to be compact surfaces, of which each connected component has boundary, and each boundary component contains a possibly zero number of ``space-like'' intervals, the remaining boundary being ``time-like''.    The morphisms are  inclusions carrying time-like boundary to time-like boundary, and each space-like boundary either to the interior, or onto a space-like boundary.

Observe that $I = T^*[-1,1]$ is an algebra object in $\mathscr{O}$:  all structure maps are given by `putting intervals next to each other', or in other words, taking cotangent bundle of the structure maps of the 1-dimensional little disks operad.  Given any object $\Sigma \in \mathscr{O}$, a choice of boundary of $\Sigma$ gives it structures of left and right $I$-module (depending on which end of $T^*[-1,1]$ we consider gluing to it). 
Similarly,  $\coprod_n T^*[-1,1]$ carries an algebra structure, for which surfaces with $n$ chosen and ordered boundary components serve as modules.

We will be interested in symmetric monoidal functors 
$\mathcal{F}: (\mathscr{O}, \sqcup) \to (\mathcal{C}, \star)$; such functors necessarily respect the algebra and module structures above.   
Recall that given an algebra object $A$ of $(\mathcal{C}, \star)$, and left and right $A$-modules $L, R \in \mathcal{C}$, 
we write $L \otimes_A R$ for the co-equalizer of the two actions of $A$ on $L \star R$; we say $\mathcal{C}$ admits tensor products if such co-equalizers always exist. 

\begin{definition} \label{def: gluing}
    By a gluing configuration, we mean a cover $\Sigma = \Sigma_L \cup \Sigma_R$ where the inclusions $\Sigma_L, \Sigma_R \to \Sigma$ are maps in $\mathscr{O}$, and the overlap is a union of strips, i.e. $A:=\Sigma_L \cap \Sigma_R \simeq \coprod T^*[-1,1]$.      
    If $(\mathcal{C}, \star)$ admits tensor products, we say $\mathcal{F}: (\mathscr{O},\sqcup) \to (\mathcal{C}, \star)$ satisfies gluing if, given any gluing configuration,  
    \begin{equation} \label{gluing} \mathcal{F}(\Sigma_L) \otimes_{\mathcal{F}(A)} \mathcal{F}(\Sigma_R) \to \mathcal{F}(\Sigma)\end{equation}
is an isomorphism.
\end{definition}

Such a functor, satisfying gluing, gives an open 2d TFT valued in the `Morita category' internal to $(\mathcal{C}, \star)$, which sends the interval to $\mathcal{F}(I)$.\footnote{There are however additional structures.  For instance, the inclusion $\emptyset \to \Sigma$ induces pointing of all structures, as Johnson-Freyd requires for a `Heisenberg-picture'  TFT  \cite{Johnson-Freyd}.} 
(We recall that the Morita category has objects given by algebras in $\mathcal{C}$, morphisms given by bimodules, and composition given by tensor of bimodules over algebras.)
We note also that while the above formulation does not directly express cutting along a non-separating strip, we may achieve the same effect by cutting along two parallel such strips. 

\begin{example}
    Homology gives a functor $(\mathscr{O}, \sqcup) \to (\Z-\mathrm{mod}, \oplus)$, which satisfies gluing by the Mayer-Vietoris theorem.
    Because the maps in $\mathscr{O}$ are proper, the same is true for Borel-Moore homology.
\end{example}

The Borel-Moore version has a known categorification.  Let us write $\Sect_{2n}$ for the category of Liouville sectors.  The objects of this category were introduced in \cite{GPS1, GPS2}: they are noncompact exact symplectic manifolds-with-boundary, satisfying certain conditions (we recall these below in Section \ref{sec: sector review}).  

\begin{example}  \label{GPS surfaces}
The `covariant functoriality' of \cite{GPS1}, specialized to the case of surfaces, 
asserts that (choosing any from the contractible space of Liouville sector structure on $\Sigma$,\footnote{Or in other words, inverting the equivalence  of topological categories $\Sect_2 \to \mathcal{O}$.} and) taking partially wrapped Fukaya category  gives a  functor $\Fuk: (\mathscr{O}, \sqcup) \to (\mathrm{Cat}, \oplus)$,  where $\mathrm{Cat}$ is any appropriate category of stable linear categories.\footnote{We work throughout in the modern context of $(\infty,1)$-categories, for which the foundational works are \cite{lurie-topos, lurie-algebra}. }  The `sectorial descent' of \cite{GPS2}, again specialized, asserts that the above $\Fuk$ satisfies gluing. 
\end{example}

One can view the present article as providing an `exponentiated' version of Example \ref{GPS surfaces}. 

\vspace{2mm}

We will write $\Sect_{2n}$ for what the casual reader can imagine is the topological category whose objects are  Liouville sectors, and whose morphisms are proper inclusions satisfying certain conditions, as described in \cite{GPS1,GPS2}.  More precisely, we will take the following  simplicial variant: objects are Liouville sectors, 1-morphisms $W_0 \to W_1$ are symplectic fibrations of sectors $\pi: \widetilde{W} \to T^*[0,1]$ with an identification $\pi^{-1}(1) = W_1$ and sectorial embedding $W_0 \to \pi^{-1}(0)$, and similarly for higher simplices.  Such a setup was described in \cite{Lazarev-Sylvan-Tanaka}.  As noted there, the functoriality, invariance, and K\"unneth results established in \cite{GPS1, GPS2} imply that taking partially wrapped Fukaya category factors through such  $\mathrm{Sect}_{2n}$. 

We recall also that the product of sectors can canonically (up to contractible choices) be given the structure of a sector; we regard this as a product $\Sect_{2n} \times \Sect_{2m} \to \Sect_{2n+2m}$.  
We will consider the category  $\Sect^\N := \prod_n \Sect_{2n}$;
an object is a sequence of (a priori unrelated) Liouville sectors, one of each dimension $X = (X_0, X_1, \ldots)$. 
Cartesian product induces a symmetric monoidal structure $(X \times Y)_n = \coprod_{a+b = n} X_a \times Y_b$. 

The set-theoretic symmetric power of a surface with boundary is not even a manifold with boundary, let alone a Liouville sector.  Nevertheless, we show: 

\begin{theorem} \label{thm: functor}
    To each $W \in \Sect_2$, we may associate a symmetric power $\Sym^{(n)}(W) \in \Sect_{2n}$, agreeing with the set-theoretic symmetric power for surfaces without boundary. 
    This assignment extends to a functor 
    $\Sym: (\mathscr{O}, \sqcup) \to (\Sect^{\N}, \times)$
     given on objects by $\Sigma \mapsto \prod_n \Sym^{(n)}(\Sigma)$. The resulting sectors are Weinstein.
\end{theorem}

Consider now a gluing configuration $\Sigma= \Sigma_L \cup_A \Sigma_R$.  Observe that there is a diagram in $\mathscr{O}$: 
\begin{equation} \label{bar diagram} \cdots \Sigma_L \sqcup A \sqcup A \sqcup \Sigma_R \,\, \substack{\rightarrow\\[-.6ex] \rightarrow \\[-.6ex] \rightarrow}   \,\, \Sigma_L \sqcup A \sqcup \Sigma_R \,\, \substack{\rightarrow\\[-.6ex] \rightarrow } \,\, \Sigma_L \sqcup \Sigma_R. \end{equation}
The maps are the action maps for the algebra object $A$ on itself and on the left and right modules $\Sigma_L$ and $\Sigma_R$.  
We will refer to this as the bar diagram for the gluing configuration, and indicate it as $\mathrm{Bar}(\Sigma_L|A|\Sigma_R)$. 
Given a cover of a topological space, $U = \bigcup U_i$, we will refer to the diagram formed by inclusions of overlaps as the \v{C}ech diagram for the cover, and denote it $\mathrm{\check{C}ech}(\bigcup U_i)$.

\begin{theorem} \label{thm: intro cover}
    Given a gluing configuration $\Sigma= \Sigma_L \cup_A \Sigma_R$, there is a (choice of Liouville form and) sectorial cover $\Sym^{(n)}(\Sigma) = \bigcup_{n_L + n_R =n} U_{n_L, n_R}$, and a (homotopy coherent) morphism of diagrams in $\Sect_{2n}$:
    $$\mathrm{\check{C}ech}(\bigcup U_{n_L, n_R}) \to \Sym^{(n)}(\mathrm{Bar}(\Sigma_L|A| \Sigma_R));$$
    moreover, each map (of sectors) between corresponding terms in the $\mathrm{Bar}$ and $\mathrm{\check{C}ech}$ diagrams is  deformation equivalent to an isomorphism. 
\end{theorem}

The basic idea behind Theorems \ref{thm: functor} and \ref{thm: intro cover}  is the following: (1) sectorial covers for cotangent bundles can be obtained  by lifting covers of the base, (2) the identification $\C = T^*\R$ gives a map $\Sym^n \C \to \Sym^n \R$ which is a cotangent bundle away from the diagonal, which (3) as we will see, locally near each strata of the diagonal, is a product of a cotangent bundle and other factors, so (4) as long as our covers are `locally pulled back from the base of the cotangent bundle factor', they will be sectorial, and (5) there is a cover of $\Sym^n \R$ with the desired combinatorics.   

\begin{remark}    Xinle Dai has previously constructed a sectorial decomposition the second symmetric power of a surface \cite{Dai-sectorial}.  
The methods are  different, and, as pointed out to us by Denis Auroux, the resulting cover is different from (the $n=2$ case) of ours: both are three element covers, but ours contains a triple intersection, corresponding to $\Sigma_L \sqcup A \sqcup A \sqcup \Sigma_R$, and Dai's does not.  
\end{remark}

We now extract algebraic consequences. 
Let $k$ be a ring  (other coefficients for which the Fukaya categories of symmetric powers can be defined would also work).  We write  $\mathrm{St}_k^{\N}$ for the category of $\N$-graded stable presentable $k$-linear categories, i.e. objects are of the form $\mathcal{C} = \bigoplus_{n} \mathcal{C}_n$, and $\otimes$ for the graded version of the Lurie tensor, i.e. $(\mathcal{C} \otimes \mathcal{D})_n = \bigoplus_{a+b = n} \mathcal{C}_a \otimes \mathcal{D}_b$.\footnote{Apparently sometimes called `Day convolution for the  monoidal structure of addition on the natural numbers'.}   Our objects will always have $\mathcal{C}_0 = k-\mathrm{mod}$. 
It follows from the `K\"unneth' considerations in \cite{GPS2} that taking Fukaya category induces a symmetric monoidal functor 
$$\Fuk: (\Sect^{\N}, \times) \to (\mathrm{St}_k^{\N}, \otimes).$$
(We regard Fukaya categories, which originate as $A_\infty$ categories, as elements of $\mathrm{St}_k$ by passing to the (dg) category of modules and then taking the nerve.  We prefer to work in the dg or stable setting to avoid confusion around the tensor product of $A_\infty$ categories.)

\begin{theorem} \label{thm: gluing}
    $\Fuk \circ \Sym :  (\mathscr{O}, \sqcup) \to (\mathrm{St}_k^{\N}, \otimes)$ satisfies gluing.  
\end{theorem}
\begin{proof}
    Consider a gluing configuration. 
    Then (the $n$'th piece of) $\Fuk(\Sym(\Sigma_L)) \otimes_{\Fuk(\Sym(A))} \Fuk(\Sym(\Sigma_R))$ is computed by the colimit of the bar diagram
    $\Fuk(\Sym^{(n)}(\mathrm{Bar}(\Sigma_L|A| \Sigma_R)))$.  Theorem \ref{thm: intro cover} identifies this bar diagram of Fukaya categories with the Cech diagram of Fukaya categories for a sectorial cover, which, by sectorial descent \cite{GPS2}, agrees with the Fukaya category of the total space $\Sym^{(n)}(\Sigma)$. 
\end{proof}

\begin{remark} \label{rem:homology}
    The same argument with sectorial descent replaced with Mayer-Vietoris shows that 
    \begin{eqnarray*}
        H_* \Sym : (\mathscr{O}, \sqcup) & \to & (k-\mathrm{mod}, \otimes)
        \\
        \Sigma & \mapsto & \bigoplus_n H_*(\Sym^{(n)}(\Sigma))
    \end{eqnarray*} satisfies gluing.  Because maps in $\Sect$ are proper embeddings, we may also take Borel-Moore homology.  
    Sending a Lagrangian to its fundamental class gives a natural transformation 
    $K(\Fuk(\Sym(\Sigma))) \to H_*^{BM}\Sym(\Sigma)$, which is injective by \cite{Lazarev-injective}.  
\end{remark}

\begin{remark}
    Theorem \ref{thm: gluing} as stated holds with coefficients in 2-periodic $\Z$-modules, since for this we need no additional structures on the $\Sym(\Sigma)$ in order to define Fukaya categories.  For another desired coefficient symmetric monoidal category $\mathcal{C}$ over which  
    the Fukaya category can be defined and satisfies sectorial descent, e.g. (not-periodized) $\Z$-modules,\footnote{The arguments in \cite{GPS2} are `geometric', and hence expected to be valid over any coefficients over which the   Fukaya category can be defined. (So, someday, spectra.)} Theorem \ref{thm: gluing} holds as stated, with the same proof, if one understands $\Fuk(\Sym(\Sigma))$ to be a family of categories over the space of $\mathcal{C}$-Maslov data\footnote{$\mathcal{C}$-Maslov data on a symplectic manifold $X$ is a null-homotopy of a certain tautological map $X \to B^2 Pic (\mathcal{C})$; it is the expected data needed to define the Fukaya category over $\mathcal{C}$.  See discussions in e.g.  \cite{lurie-rotation, nadler-shende, GPS3, CKNS, Gavela-Large-Ward}.} for $\prod \Sym^n (\Sigma)$.  
    Alternatively, we could formulate the result in terms of a functor from the category $\mathscr{O}_{\mathcal{C}}$ given by pulling back the space of $\mathscr{C}$-Maslov data along $\Sym: \mathscr{O} \to \Sect^\N$.   
\end{remark}

\begin{remark}
    It is clear from the proof that $\Fuk$ may be replaced by any functor which satisfies sectorial descent.
    Microsheaves offer one such category \cite{shende-microlocal, nadler-shende}, defined over any stable compactly generated symmetric monoidal category, e.g. spectra; and equivalent to  the Fukaya category when the latter is  defined \cite{GPS3}.
\end{remark}

We turn to questions of calculation.  Various models for $\Fuk(\Sym(\Sigma))$ and related categories have been obtained in previous works \cite{Lipschitz-Ozsvath-Thurston,  Auroux-bordered,
Douglas-Manolescu, Douglas-Lipschitz-Manolescu, rouquier-manion, Dyckerhoff-Jasso-Lekili, didedda}.  We will not attempt to rederive all of these here, but will explain how to recover the main result of Rouquier and Manion  \cite[Thm. 1.3.1]{rouquier-manion}, which asserts that (1) if $\Sigma$ is a surface with chosen interval boundary, $\Fuk(\Sym(\Sigma))$ is a representation of a categorification of $\mathfrak{gl}(1|1)^+$ and (2)
    there is a certain explicit monoidal structure $\ootimes$ on such representations, such that if $\Sigma$ is obtained by gluing some $\Sigma', \Sigma''$ to an open pair of pants, then 
    \begin{equation} \label{rouquier-manion formula}\Fuk(\Sym(\Sigma)) = \Fuk(\Sym(\Sigma')) \,\, \ootimes \,\Fuk(\Sym(\Sigma'')).\end{equation}
    Let us explain what this assertion looks like from our point of view. 
    Recall we write $I = T^*[-1,1] \in \mathscr{O}$.  
    Let us write $T$ for the open pair of pants, i.e., the  disk with three stops; it is an $I-(I \coprod I)$ bimodule. Given any functor $\mathcal{F}$ satisfying gluing, 
    the different ways of decomposing the disks with $n$ boundaries into pairs of pants give the associativity conditions showing that $\mathcal{F}(T)$ defines a lax coalgebra structure for $\mathcal{F}(I)$, and in particular an ($E_1$) monoidal structure $\star_T$ on the category of $\mathcal{F}(I)$-modules.     
    
    Since Theorem \ref{thm: gluing} promises that $\Fuk \circ \Sym$ satisfies gluing, we immediately recover  \eqref{rouquier-manion formula} with the categorified $\mathfrak{gl}(1|1)^+$ and $\ootimes$ replaced by 
    $\Fuk(\Sym(I))$ and $\star_T$.
    
    For a precise comparison, we use (for the first time in this article) actual knowledge of  $\Fuk(\Sym(I))$: 
    \begin{eqnarray*}
        \Fuk(\Sym^0(I)) & = & \Fuk(\bullet) = k-\mathrm{mod} \\
        \Fuk(\Sym^1(I)) & = & \Fuk(I) = k-\mathrm{mod} \\
        \Fuk(\Sym^{\ge 2}(I)) & = & 0.
    \end{eqnarray*}
    That is, $\Fuk\Sym(I)$ is a  monoidal category generated by a free degree one category, with monoidal square zero.\footnote{The categorified $\mathfrak{gl}(1|1)^+$ of  \cite{khovanov-gl12,
    rouquier-manion} is by design the monoidal category whose modules carry an endomorphism with a null-homotopy of its square.  If we take its modules in the world of stable categories (where we work) -- where a morphism being zero is a property, not a structure -- then said module category is equivalent to that of our $\Fuk\Sym(I)$.} 
    Graded modules for this monoidal category are, tautologically, sequences of stable k-linear categories $\cdots \to C^{-1} \to C^0 \to C^1 \to \cdots$, where the composition of consecutive maps is zero.  
    Such  `categorical complexes' were recently studied by Christ, Dyckerhoff, and Walde \cite{Christ-Dyckerhoff-Walde}.
    
    A first appealing feature of this correspondence:  as the Fukaya categories of all nonzero symmetric powers of a half disk vanish, one can easily check  
     $\Fuk\Sym(D) \otimes_{\Fuk\Sym(I)} \cdot : \Fuk\Sym(I)-\mathrm{mod} \to \mathrm{St}_k$ is a `semi-cohomology': it sends 
    a complex $(C,d)$ as above to the graded category $(C / \mathrm{Image}(d))$.  
    Going further, the operation  $\star_T$ must correspond to some monoidal structure which takes two categorical complexes and returns a single one.     
    In fact \cite{Christ-Dyckerhoff-Walde} provides just such an operation: totalization of the tensor product bicomplex, which is given by a much more reasonably sized pushout, and can be checked to be the translation to stable categories of the  $\ootimes$ of Rouquier and Manion (see Remark \ref{rem: RM versus CDW} below).
    We show, by algebra:  

    \begin{theorem} \label{thm:cdw} (\ref{thm:tensor to pushout})
        The identification of $\Fuk(\Sym(I))$-modules with categorical complexes carries the monoidal structure $\star_T$ to the coproduct totalization of \cite{Christ-Dyckerhoff-Walde}.  
    \end{theorem}

    We also give a  totalization for a (singly graded) categorical complex carrying two commuting degree one differentials, and prove it matches the  gluing of two ends of the pair of pants to a not-necessarily-disconnected surface with two marked boundaries, or more generally, to an arbitrary $\Fuk(\Sym(I)) \otimes \Fuk(\Sym(I))$-module.  

    Any surface can be assembled by attaching strips to pairs of pants and capping by a half disk.  And so:
    
    \begin{corollary}
        For any $\Sigma \in \mathscr{O}$, the category $\Fuk(\Sym(\Sigma))$ is the semicohomology of an iterated totalization of a tensor power of $\Fuk(\Sym(I))$.
    \end{corollary}

    We conclude with a remark about   future directions.  The various symplectic approaches to Khovanov homology and its relatives \cite{Seidel-Smith, manolescu1, manolescu2, 
    abouzaid-smith-arc, abouzaid-smith:khovanov, 
    aganagic-knot-2, aganagic-icm,
    ALR,
    ADLSZ, 
    lepage-shende}
     involve spaces which admit natural maps to symmetric powers of surfaces.  In an ongoing joint project with Aganagic and LePage, we will apply a generalization of the methods developed here to construct similar sectorial decompositions for such spaces,  and construct Rouquier's monoidal structure (\cite{Rouquier-lecture}) for general $\mathfrak{g}^+$, by the analogue of the $\star_T$ above, and establish their expected properties.
    In another direction, we expect  similar `exponentiated sectorial descent' theorems can be proven in the `higher dimensional Heegaard-Floer' theory \cite{Colin-Honda-Tian}, as would be compatible with the calculation \cite{Honda-Tian-Yuan} and of presumptive representation theoretic relevance visavis \cite{Yuan-link}.

    \vspace{2mm} \noindent{\bf Acknowledgements.} 
This article originates in our ongoing work with Mina Aganagic and Elise LePage on understanding the categorified quantum groups and their representation theory via the Fukaya categories of certain string-theoretically motivated moduli spaces, and owes a great deal to discussions with them.  We  thank Mikhail Kapranov for a suggestion which led to the formulation in terms of $\mathscr{O}$.  Finally, let us also mention that Rapha\"el Rouquier asked the first-named author, circa 2018, whether the ideas around sectorial descent could be useful for constructing $\ootimes$.

        V.S. is supported by  Villum Fonden Villum Investigator grant 37814.

\vspace{2mm} \noindent
{\bf Notation.}
For any set $M$, we write $\Sym^n M$ for the space of unordered tuples, possibly with multiplicity,   
$\mathbf{z} = (z_1, \cdots, z_n)  \in \Sym^n(M)$.  We write $|\mathbf{z}|=n$ for the number of points.  For $\mathbf{z}_1 \in \Sym^{n_1}(M)$ and $\mathbf{z}_2 \in \Sym^{n_2}(M)$, we write $\mathbf{z}_1 + \mathbf{z}_2 \in \Sym^{n_1+n_2}(M)$ for the  union-with-multiplicity.

For a partition $\lambda \vdash n$, with $\lambda = (\lambda_1, \lambda_2, \ldots, \lambda_k)$,  we write $\Sym^\lambda(M) \subset \Sym^n(M)$ for the locus consisting points with multiplicities $\lambda_1, \ldots, \lambda_k$.  This is a stratification:  $$\Sym^n(M) = \coprod_{\lambda \vdash n} \Sym^\lambda(M).$$
We write $\lambda \le \mu$ when $\lambda$ is obtained by combining some of the parts of $\mu$; so $(1,1,\ldots,1)$ is the largest partition and $(n)$ the smallest.   We have:
\begin{equation}  \label{partition ordering captures closure} 
\Sym^{\le \mu}(M) := \bigcup_{\lambda \le \mu} \Sym^\lambda(M) = \overline{\Sym^\mu(M)}\end{equation}

\section{Clustering rules}\label{sec: clustering rules}

One fundamental technical tool of the present article is a system of coordinates on symmetric products which allow us to separate dependence on the  normal and tangential directions to the big diagonal.  Were the diagonal smooth (it is not), it would suffice to fix a metric and choose a tubular neighborhood.
Let us elaborate on the basic challenge.  Given a point on the big diagonal, $\mathbf{z} = \sum n_i z_i$ with $n_i > 1$, we might consider a class of nearby configurations 
$\mathbf{z} = \sum \mathbf{z}_i$, with $|\mathbf{z}_i| = n_i$ and $\mathbf{z}_i$ centered at $z_i$.  It is natural to view  variation of the $\mathbf{z}_i$ as tangential to the diagonal, and internal  variation constituents of $\mathbf{z}_i$ (constrainted to remain centered at $z_i$) as normal to the diagonal.  
Of course, each $\mathbf{z}$ admits many expressions of the form
$\mathbf{z} = \sum \mathbf{z}_i$.  In this section we will describe a certain class of such decompositions (clusterings), in order to define a class of functions on symmetric powers, which depend only on the center of masses of clusterings.

We proceed with the construction. 
Let $M$ be a smooth Riemannian manifold with distance function $\dist: M \times M \to \R$.  Recall that a subset $U \subset M$ is said to geodesically convex if any two points in $U$ can be connected by a unique geodesic, and the distance square function $\dist(-,x)^2$ restricts to any length-parametrized geodesic is a convex function.
The radius of convexity at a point $p$ is the radius of the largest geodesically  convex open ball around $p$; we write  $r_{conv}$ for the infimum over $M$ of the convexity radii. We refer to \cite[Ch 5.4]{petersen2006riemannian} for properties of the convexity radius.  We will  assume that $r_{conv}> 0$, which will hold for our examples of interest because Liouville manifolds have bounded geometry.   

\begin{definition}
    Let $\bz=(z_1, \cdots, z_k) \in \Sym^k M$, if $\dist(z_i, z_j) < r_{conv}$, we say $\bz$ is {\em pre-confined}. 
\end{definition}

\begin{lemma}\label{lm:center}
    Let $\bz=(z_1, \cdots, z_n) \in \Sym^n(M)$ be pre-confined. Let $g: M \to \R$ be the continuous function given by
    $ g(z) = \sum_{i=1}^n \dist(z_i, z)^2. $
    Fix $\rho$ such that $\max_{i,j}(\dist(z_i,z_j)) < \rho < r_{conv}$. 
    Let $U = \cap_i B_{\rho}(z_i)$, then $g$ has a unique minimum in $U$. 
\end{lemma}
\begin{proof}
Each ball $B_\rho(z_i)$ is strongly geodesically convex, hence their
intersection $U$ is geodesically convex as well: if $x,y\in U$, then for each
$i$ the unique minimizing geodesic from $x$ to $y$ lies in $B_\rho(z_i)$, hence
it lies in $U$. The functions $d(z,z_i)^2$ is smooth and strictly (geodesic) convex, hence their sum $g(z)$ is strictly convex. Thus $g(z)$ can have at most one minimum. 

Since $\overline U$ is compact and $g$ is continuous, $g$ attains its minimum
at some point \(p\in \overline U\). We claim that \(p\in U\). Suppose instead that \(p\in \partial U\). Then for
some \(j\) we have \(d(p,z_j)=\rho\), so \(p\in \partial B_\rho(z_j)\). Let
\(\nu\) be the inward unit normal to \(\partial B_\rho(z_j)\) at \(p\). Since
every \(z_i\) lies in the closed ball \( \overline{B_\rho(z_j)}\),  moving $p$ inward will decrease $g$ contradicting $p$ being the minimum. 
\end{proof}

\begin{definition}
    For pre-confined $\bz \in \Sym^n M$,  the center of mass of $\bz$ is the minimum point produced in Lemma \ref{lm:center}; we denote it $c(\bz)$.  We write $r(\bz) := \max_{i} \dist(z_i, c(\mathbf{z}))$. 
\end{definition}

When $V$ is a normed vector space, $r_{conv} = \infty$, every $\mathbf{z} \in \Sym^n M$ is pre-confined, and the center of mass is the weighted average of the $\{z_i\}$.

\begin{definition} \label{def: clustering rule}
    A {\em clustering rule} is two sequences of positive real numbers, 
    $$ 0 < r_1 < r_2  <\cdots < r_N \ll r_{conv}, \quad 0 = d_1 < d_2 < \cdots < d_N \ll r_{conv}$$
    such that 
    $$ r_k > d_k + r_{k-1}, \quad d_k > 6 r_{k-1}. $$ 
\end{definition}
\begin{lemma}
        For all $a,b \geq 1$, we have
$$ r_{a+b} > d_{a+b} + \max(r_a, r_b), \quad d_{a+b} > 3 (r_a + r_b)$$
\end{lemma}
\begin{lemma}
Clustering rules exist. 
\end{lemma}
\begin{proof}
We may choose $d_2 > 0$, $\epsilon > 0$, and set $r_1 = d_2/(6(1+\epsilon))$.  Then,  for $k = 2, 3, \cdots$, we set  
$$ r_k = (1+\epsilon) (d_k +  r_{k-1}), \quad d_{k+1} = (1+\epsilon) 6 r_{k}. $$ 
\end{proof}

\begin{definition} \label{def:clusters}
Fix a clustering rule.  We say $\mathbf{z}$ is {\em confined} if it is pre-confined and  $r(\mathbf{z}) < r_{|\mathbf{z}|}$.  We say $\mathbf{z}_1, \mathbf{z}_2$ are {\em separated} if they are pre-confined and  $\dist(c(\mathbf{z}_1), c(\mathbf{z}_2)) > d_{|\mathbf{z}_1| + |\mathbf{z}_2|}$. 
We say $\mathbf{z} = \sum \mathbf{z}_i$ is a cluster decomposition if each $\mathbf{z}_i$ is confined and each pair $\mathbf{z}_i, \mathbf{z}_j$ are separated.  We say the {\em type} of the decomposition is the partition 
$|\mathbf{z}| = \sum |\mathbf{z}_i|$.  
\end{definition}

\begin{lemma} \label{merging lemma}
    Suppose $\mathbf{z}_1, \mathbf{z}_2$ are confined, but not separated.  Then $\mathbf{z}_1 + \mathbf{z}_2$ is confined.  
\end{lemma}
\begin{proof} We compute: 
    \begin{align*}
    r(\mathbf{z}_1 + \mathbf{z}_2) &= \max_{z \in \mathbf{z}_1 + \mathbf{z}_2} \dist(z, c(\mathbf{z}_1 + \mathbf{z}_2)) = \max_{i=1,2} \max_{z \in \mathbf{z}_i} \dist(z, c(\mathbf{z}_1 + \mathbf{z}_2)) \\
    & < \max_{i=1,2} \max_{z \in \mathbf{z}_i} \dist(z, c(\mathbf{z}_i)) + \dist(c(\mathbf{z}_i),c(\mathbf{z}_1 + \mathbf{z}_2)) \\
    & < \max(r_{|\mathbf{z}_1|}, r_{|\mathbf{z}_2|}) + d_{|\mathbf{z}_1| + |\mathbf{z}_2|} \\
    &< r_{|\mathbf{z}_1 + \mathbf{z}_2|}\\
\end{align*}  
\end{proof}

\begin{corollary}
Every $\mathbf{z} \in \Sym^n(M)$ admits a cluster decomposition.
\end{corollary}
\begin{proof}
    If $\mathbf{z} = \sum n_i z_i$, begin with the decomposition $\mathbf{z} = \sum \mathbf{z}_i$ where $\mathbf{z}_i = n_i z_i$.      
    Iteratively merge pairs $\mathbf{z}_i, \mathbf{z}_j$ which are confined but not separated; at each stage, by Lemma \ref{merging lemma}, we maintain a decomposition into pieces which are confined.  By finiteness of $|\mathbf{z}|$ the procedure terminates (possibly at $\mathbf{z} = \mathbf{z}$ if $\mathbf{z}$ itself is confined). 
\end{proof}
Let $\bz_1, \bz_2$ be two sets of points in $M$. We define their {\em separation} as
$$ \sep(\mathbf{z}_1, \mathbf{z}_2) := \min_{z \in \mathbf{z}_1, w \in  \mathbf{z}_2 } \dist(z, w).$$

\begin{lemma} \label{lem: cluster distance bound}
    For any cluster decomposition $\mathbf{z} = \sum \mathbf{z}_i$, we have
    $$ \sep(\bz_i,\bz_j) > 2 (r_{|\mathbf{z}_i|} +  r_{|\mathbf{z}_j|}) $$
\end{lemma}
\begin{proof}
For any $z \in \mathbf{z}_i, w \in \mathbf{z}_j$, we have
\begin{align*}
    \dist(z,w) & > \dist(c(\mathbf{z}_i),  c(\mathbf{z}_j)) - \dist(z, c(\mathbf{z}_i)) - \dist( w, c(\mathbf{z}_j)) \\
    & >  d_{|\mathbf{z}_i|+|\mathbf{z}_j|} - r_{|\mathbf{z}_i|} - r_{|\mathbf{z}_j|} \\ &
    > 2 r_{|\mathbf{z}_i|} + 2 r_{|\mathbf{z}_j|}.
\end{align*}
\end{proof}
In particular, the disks around the centers of clusters which witness the confinement are disjoint.  Even more particularly: if $\mathbf{z}= \sum n_i z_i$ is the sum with multiplicities, then if $z_i$ appears in some cluster in a cluster decomposition, said cluster in fact contains $n_i z_i$.   Because of this, if $\mathbf{x}, \mathbf{y}$ are clusters in possibly different cluster decompositions of some given $\mathbf{z}$, it makes unambiguous sense to write $\mathbf{x} \cap \mathbf{y}$, $\mathbf{x} \setminus \mathbf{y}, \mathbf{y} \setminus \mathbf{x}$. 

\begin{lemma} \label{comparing cluster decompositions}
    Suppose $\mathbf{x}$ and $\mathbf{y}$ are clusters in possibly different cluster decompositions of $\mathbf{z}$.  Then either $\mathbf{x}, \mathbf{y}$ are disjoint, or one contains the other. 
\end{lemma}
\begin{proof}
The assertion is equivalent to asking that not all three of $\mathbf{x} \setminus \mathbf{y}, \mathbf{y} \setminus \mathbf{x}, \mathbf{x} \cap \mathbf{y}$ are all nonempty. 

Suppose $\mathbf{x} \setminus \mathbf{y}$ and $\mathbf{x} \cap \mathbf{y}$ are both nonempty.  We have: 
$$ 2 r_{|\mathbf{x}|} > 2 r(\mathbf{x}) > \sep(\mathbf{x} \setminus \mathbf{y}, \mathbf{x} \cap \mathbf{y}) \geq \sep(\mathbf{x} \setminus \mathbf{y}, \mathbf{y}) > 2r_{|\mathbf{y}|},$$
where we have used  Lemma \ref{lem: cluster distance bound} for the final inequality.

Similarly, if $\mathbf{y} \setminus \mathbf{x}$ and $\mathbf{x} \cap \mathbf{y}$ are both nonempty, then $2r_{|\mathbf{y}|} > 2 r_{|\mathbf{x}|}$. 
Thus if all three are nonempty, we arrive at a contradiction. 
\end{proof}

\begin{corollary} \label{cluster decompositions common coarsening}
    Any two cluster decompositions of a given $\mathbf{z}$ have a common coarsening.
\end{corollary}
\begin{proof}
    Suppose given cluster decompositions $\mathbf{z} = \sum \mathbf{x}_i = \sum \mathbf{y}_j$.  Let $\mathbf{w}_k$ be the maximal elements for containment amongst the $\mathbf{x}_i, \mathbf{y}_j$.  By Lemma \ref{comparing cluster decompositions}, in fact  $\mathbf{z} = \sum \mathbf{w}_k$.  Since the $\mathbf{w}_k$ were among the $\mathbf{x}_i, \mathbf{y}_j$, they are confined; in case they are not separated, we may  iteratively apply Lemma \ref{merging lemma} to eventually obtain the desired coarsening.
\end{proof}

\begin{proposition} \label{cluster decompositions common refinement}
    Any two cluster decompositions of a given $\mathbf{z}$ have a common refinement.  In particular, there is a finest cluster decomposition.
\end{proposition}
\begin{proof}
Suppose given cluster decompositions $\mathbf{z} = \sum \mathbf{x}_i = \sum \mathbf{y}_j$.  The claim is that $\mathbf{z} = \sum \mathbf{x}_i \cap \mathbf{y}_j$ 
is a cluster decomposition.
The result is obvious if one of the $\mathbf{x}_i$ and $\mathbf{y}_j$ decomposition is a refinement of the other; so suppose otherwise.  

By Lemma \ref{comparing cluster decompositions}, each $\mathbf{x}_i \cap \mathbf{y}_j$ is either $\mathbf{x}_i$ or $\mathbf{y}_j$ or empty, hence in particular satisfies confinement.  
We check the separation constraint for the $\mathbf{x}_i \cap \mathbf{y}_j$.  Of course it holds already for pairs of clusters from $\mathbf{x}$ or $\mathbf{y}$.  The remaining possibility is that there is say $\mathbf{x}_1 \subsetneq \mathbf{y}_{1'}$ and $\mathbf{x}_2 \supsetneq \mathbf{y}_{2'}$. 
We compute:
$$ \dist(c(\mathbf{x}_1) , c(\mathbf{y}_{2'})) > \dist(c (\mathbf{x}_1), c(\mathbf{x}_2) ) - \dist(c(\mathbf{x}_2), c(\mathbf{y}_{2'})) >  d_{|\mathbf{x}_1 + \mathbf{x}_2|} - r_{|\mathbf{x}_2|} > d_{|\mathbf{x}_1 + \mathbf{x}_2|} - r_{|\mathbf{x}_1|+|\mathbf{x}_2|-1}.$$
Then using $d_{|\mathbf{x}_1 + \mathbf{x}_2|} > 3 r_{|\mathbf{x}_1|+|\mathbf{x}_2|-1}$ and $r_{|\mathbf{x}_1|+|\mathbf{x}_2|-1} > d_{|\mathbf{x}_1|+|\mathbf{x}_2|-1}$, we have
$$ \dist(c(\mathbf{x}_1),c(\mathbf{y}_{2'}))| > 2 d_{|\mathbf{x}_1|+|\mathbf{x}_2|-1} \ge 2 d_{|\mathbf{x}_1| + |\mathbf{y}_{2'}|} > d_{|\mathbf{x}_1| + |\mathbf{y}_{2'}|}. $$
\end{proof}

\begin{lemma}
    Suppose $\mathbf{z}_1$ and $\mathbf{z}_2$ are each confined, and $(\mathbf{z}_1, \mathbf{z}_2)$ is separated.  Then for any confined $\mathbf{z}_1' \subset \mathbf{z}_1$ and $\mathbf{z}_2' \subset \mathbf{z}_2$, the pair $(\mathbf{z}_1', \mathbf{z}_2')$ is separated.  
\end{lemma}
\begin{proof}
    Suppose not. Then: 
    $$d_{|\mathbf{z}_1'| + |\mathbf{z}_2'|} > \dist(c(\mathbf{z}_1') ,c(\mathbf{z}_2')) \ge
    \dist(c(\mathbf{z}_1) , c(\mathbf{z}_2)) - 
    \dist(c(\mathbf{z}_1) , c(\mathbf{z}_1')) - \dist(c(\mathbf{z}_2), c(\mathbf{z}_2')) > d_{|\mathbf{z}_1| + |\mathbf{z}_2|} - r_{|\mathbf{z_1}|} - r_{|\mathbf{z}_2|}$$
    This is a contradiction (recall $d_k > 6 r_{k-1} > 6d_{k-1}$). 
\end{proof}

\begin{corollary}
    If $\mathbf{z} = \sum \mathbf{z}_i$ is a cluster decomposition, and $\mathbf{z}_i = \sum \mathbf{z}_{i, \alpha}$ are cluster decompositions, then $\mathbf{z} = \sum \mathbf{z}_{i,\alpha}$ is a cluster decomposition. 
\end{corollary}

More generally: 

\begin{corollary}\label{contraction induces merging}
    Let $f: M \to N$ be a contraction, i.e. $\dist(f(z_1), f(z_2)) \leq \dist(z_1, z_2)$ for any $z_1, z_2 \in M$. Let $f: \Sym^n M \to \Sym^n N$ be the natural extension.  Suppose given $\mathbf{z} \in \Sym^n M$ and a (not necessarily cluster) decomposition $\mathbf{z} = \sum_i \mathbf{z}_i$.  If  $f(\mathbf{z}) = \sum_i f(\mathbf{z}_i)$ is a cluster decomposition, and $\mathbf{z}_i = \sum_\alpha \mathbf{z}_{i, \alpha}$ are cluster decompositions, then $\mathbf{z} = \sum_{i, \alpha} \mathbf{z}_{i,\alpha}$
    is a cluster decomposition.  
\end{corollary}

For any partition $\tau$, we write $\Clus^\tau(M) \In \Sym^n(M)$ for the (open) subset of elements admitting a type $\tau$ decomposition.  We write  $\Clus^{\leq \tau}(M) = \cup_{\sigma \leq \tau} \Clus^\tau (M) $.

\begin{proposition}
For any partition $\tau$,  $\Sym^{\tau}(M) \subset  \Clus^{\le \tau} (M)$. 
\end{proposition}
\begin{proof}
Given any $\mathbf{z} \in \Sym V$, we  initialize a candidate cluster decomposition, where only coincidental points form a cluster. This initial candidate satisfies confinement constraints, but possibly not the separation constraint. 
We will iteratively merge clusters if they are not separated, i.e. $\dist(c(\mathbf{z}_1), c(\mathbf{z}_2)) < d_{|\mathbf{z}_1| + |\mathbf{z}_2|}$; per Lemma \ref{merging lemma}, the result consists again of a decomposition into candidate clusters which are all confined. 
Eventually we arrive at a cluster decomposition of a type obtained by merging parts of $\tau$, i.e. of a type $\le \tau$. 
\end{proof}

\begin{corollary}
    $\Sym^{\le \tau}(M) \subset \Clus^{\le \tau}(M)$. 
\end{corollary}

\begin{lemma} \label{relax r and d}
If $(r_i,d_i)$ is a set of clustering rules, then there exists another set of clustering rules $(r_i', d_i')$, such that $r_i' > r_i$ and $d_i'<d_i$ for all $i$. 
\end{lemma}
\begin{proof}
    For any fixed $k$, if we fix all parameters $(r_i, d_i)$ except $r_k$, there are finitely many open constraints for $r_k$, e.g., 
    $$  r_k > d_k + r_{k-1}, \quad r_{k+1} > d_{k+1} + r_k, \quad d_{k+1} > 6 r_k. $$
    Similarly, if we fix all other parameters except $d_k$, there are only finitely many open constraints for $d_k$. Thus, we can initialize $r_i'=r_i, d_i'=d_i$, then adjust them iteratively one by one for $i=1,2,\cdots$, such that $r_i' > r_i$ and $d_i'<d_i$ for all $i$. 
\end{proof}

\begin{definition}
Fix a clustering rule.  A function $f: \Sym^n(M) \to X$ is a {\em center of mass function} if there exist functions $f_\tau: \Sym^\tau(M) \to X$ such that for any cluster decomposition
$\mathbf{z} = \mathbf{z}_1 + \ldots + \mathbf{z}_k$ of type $\tau$, we have
$$f(\mathbf{z}) = f_\tau(|\mathbf{z}_1| c(\mathbf{z}_1) + \ldots + |\mathbf{z}_k| c( \mathbf{z}_k) )$$
\end{definition}

Note the definition is local on $\Sym^n(M)$, so makes sense for functions defined only on some open subset, and indeed such functions form a sheaf.  One obvious nonconstant such global function is the total center of mass. We will construct some useful center of mass functions in Lemma \ref{box opens have cm boundary} below.   

\begin{remark}
    Continuous functions on $\Sym^n \C$ which pull back to holomorphic functions on $\C^n$ are already holomorphic on $\Sym^n \C$ (fundamental theorem of symmetric functions), but  continuous functions on $\Sym^n \C$ which pull back to smooth functions on $\C^n$ need not be smooth.  However, a continuous function on $\Sym^n \C$ which is a center of mass function and which pulls back to a smooth function on $\C^n$ is already smooth on $\Sym^n \C$.  
\end{remark}

\section{K\"ahler potentials} \label{sec: kp}

In this section we construct K\"ahler potentials on symmetric products of surfaces which are well adapted to studying gluing configurations. 

The smoothing of K\"ahler potential on $\Sym^n \Sigma$ for smooth affine $\Sigma$ was considered by Perutz \cite[Section 7]{perutz2008hamiltonian}, based on smoothing and patching technique of Richberg \cite{richberg1967stetige} and Varouchas \cite{varouchas1984stabilite}. Here we use the same tools but exert more control of the smoothing over  strip-like regions $T^*[-2,2] \In \Sigma$, in order to prove sectoriality later.

\subsection{Recollections on smoothing plurisubharmonic functions}

Recall that a
smooth function $f$ is said to be (strictly) subharmonic if   
$\Delta f \ge 0$ (resp. $\Delta f > 0$).\footnote{Many texts allow also upper semi-continuous $f$, where Laplacian is understood in a distributional sense, ie, $\int \Delta f \cdot g \geq 0$ for all compactly supported non-negative test functions $g$.  Here we will only need continuous psh functions.}
A function on a complex manifold, $\phi: X \to \R$, is said to be (strictly) plurisubharmonic for any holomorphic map from a 1-complex-dimensional disk, $i: D \to X$, the pullback $i^* \phi$ is (strictly) subharmonic.  Note that if $\pi: X \to Y$ is any finite map of complex manifolds and $\phi$ is strictly plurisubharmonic on $X$, then $\pi_* \phi$ is strictly plurisubharmonic on $Y$ (we will later apply this to $\C^n \to \Sym^n \C$). 

In case $\phi$ is smooth, subharmonicity is equivalent to asking that the complex Hessian $\omega_\phi :=  i \partial \bar \partial \phi$ is positive semidefinite, and strict plurisubharmonicity
to positive definiteness.  A smooth strictly plurisubharmonic function is referred to as a K\"ahler potential.  A discussion of plurisubharmonic functions in the context of symplectic geometry can be found in \cite{cieliebak-eliashberg}. If we define $d^c =  (\d - \bar \d)/(2i)$, then $ d d^c = i \d \bar \d$, and we get
$$ \lambda = d^c \phi, \quad \omega = d \lambda = dd^c \phi = i \d \bar \d \phi. $$

\begin{definition}
    Let $\phi$ be a continuous 
    strictly plurisubharmonic function on $X$. 
    We say $(\varphi_t)_{t \in (0, t_0)}$ is a smoothing of $\phi$,  if for each $t \in (0, t_0)$, $\varphi_t$ is smooth strictly plurisubharmonic on $X$ and   $ \|\varphi_{n,t} - \phi_n\|_{C^0(X)} < t$.
\end{definition}

A theorem of Richberg asserts the existences of psh smoothings of continuous psh functions \cite{richberg1967stetige}: 

\begin{theorem}[\protect{\cite{richberg1967stetige},  
\cite[Thm 5.21]{demailly1997complex}}]\label{thm:richberg}
Let $u \in \mathrm{Psh}(X)$ be a continuous function which is strictly
plurisubharmonic on an open subset $\Omega \subset X$, with
$
Hu \ge \gamma
$
for some continuous positive Hermitian form $\gamma$ on $\Omega$.
For any continuous function $\lambda \in C^0(\Omega)$ with $\lambda > 0$,
there exists a plurisubharmonic function
$
\tilde u \in C^0(X) \cap C^\infty(\Omega)
$
such that
$
u \le \tilde u \le u + \lambda \quad \text{on } \Omega,
$
and $\tilde u = u$ on $X \setminus \Omega$. Moreover, $\tilde u$ is
strictly plurisubharmonic on $\Omega$ and satisfies
$
H\tilde u \ge (1-\lambda)\gamma.
$
In particular, $\tilde u$ can be chosen strictly plurisubharmonic on $X$
if $u$ has the same property.
\end{theorem}

The regularized max function is useful for patching:

\begin{lemma}\cite[Lemma 5.18]{demailly1997complex} \label{lm: reg-max}
Let $\theta \in C^\infty(\mathbb{R},\mathbb{R})$ be a nonnegative function with support in $[-1,1]$ such that
$\int_{\mathbb{R}} \theta(h)\,dh = 1$ and $\int_{\mathbb{R}} h\,\theta(h)\,dh = 0$.
For arbitrary $\eta = (\eta_1,\ldots,\eta_p) \in (0,+\infty)^p$, the function
\begin{equation}
M_\eta(t_1,\ldots,t_p)
=
\int_{\mathbb{R}^p}
\max\{t_1 + h_1,\ldots,t_p + h_p\}
\prod_{1 \le j \le p} \theta(h_j/\eta_j)
\, dh_1 \cdots dh_p
\end{equation}
possesses the following properties:
\begin{enumerate}
\item $M_\eta(t_1,\ldots,t_p)$ is non decreasing in all variables, smooth and convex on $\mathbb{R}^p$;

\item $\max\{t_1,\ldots,t_p\} \le M_\eta(t_1,\ldots,t_p)
\le \max\{t_1+\eta_1,\ldots,t_p+\eta_p\}$;

\item $M_\eta(t_1,\ldots,t_p)
= M_{(\eta_1,\ldots,\widehat{\eta_j},\ldots,\eta_p)}
(t_1,\ldots,\widehat{t_j},\ldots,t_p)$
if $t_j+\eta_j \le \max_{k\ne j}\{t_k-\eta_k\}$;

\item $M_\eta(t_1+a,\ldots,t_p+a)=M_\eta(t_1,\ldots,t_p)+a$,
$\quad \forall a\in\mathbb{R}$;

\item if $u_1,\ldots,u_p$ are plurisubharmonic and satisfy
$H(u_j)_z(\xi)\ge\gamma_z(\xi)$ where $z\mapsto\gamma_z$
is a continuous hermitian form on $T_X$, then
$u=M_\eta(u_1,\ldots,u_p)$ is plurisubharmonic and satisfies
$H u_z(\xi)\ge\gamma_z(\xi)$.
\end{enumerate}
If $\eta_1=\cdots=\eta_p=\delta$, we will write $M_\delta$.
\end{lemma}

The following is a variant of a lemma of Varouchas \cite{varouchas1984stabilite}, which allows to patch together two locally defined smoothings:

\begin{lemma} \label{Varouchas lemma}
Let $X$ be a complex manifold, $\phi$ a  continuous strictly plurisubharmonic function on $X$. Let $X = V \cup W$ be an open cover, $W' = X \RM \bar V$ and $V' = X\RM \bar W$. Let $\varphi^V_t$ and $\varphi^W_t$ be smoothing family of $\phi|_V$ and $\phi|_W$. Then there exists a smoothing family $\varphi_t$ on $X$, for small enough $t$, such that $\varphi_t = \varphi^V_t$ on $V'$ and $\varphi_t = \varphi^W_t$ on $W'$. 
\end{lemma}
\begin{proof} 
Let $1 = \eta^V + \eta^W$ be a smooth partition of unity, so that $\eta^V|_{V'}=1$ and $\eta^W|_{W'}=1$. 
For small enough $s$, and $t(s)$ small enough depending on $s$, such that  $t(s) < s/4$, and
$s \eta^V + \varphi^V_{t(s)}$ and $s \eta^W + \varphi^W_{t(s)}$ are still strictly plurisubharmonic on $V$ and $W$, respectively. We consider, for small enough $s$, 
$$ \varphi_{t(s)}(z)= \begin{cases}
    \varphi^V_{t(s)} & z \in \bar V' \\
    M_{s/4}(s \eta^V + \varphi^V_t, s \eta^W + \varphi^W_t)  & z \in V \cap W \\
    \varphi^W_{t(s)} & z \in \bar W' 
\end{cases}
$$
Near $\d V$, we have 
$$ |  (s \eta^V + \varphi^V_t) - (s \eta^W + \varphi^W_t) | = |(-s  + \varphi^V_t) - \varphi^W_t| > s - |\varphi^V_t - \varphi^W_t| > s - 2t > s/2, $$
hence $M_{s/4}(s \eta^V + \varphi^V_t, s \eta^W + \varphi^W_t) = \varphi^W_t$ (Lemma \ref{lm: reg-max}, (3)), hence $\varphi_{t(s)}(z)$ smoothly transit to $\varphi^W_t$. Similarly near $\d W$, $\varphi_{t(s)}(z)$ smoothly transit to $\varphi^V_t$. This shows $\varphi_{t(s)}(z)$ is globally defined and psh (Lemma \ref{lm: reg-max}, (5)). 
\end{proof}

\subsection{Factorizable K\"ahler potentials on symmetric powers}
\def\bz{\mathbf{z}}

Fix a real quadratic form $\phi: \C = \R^2 \to \R$ with positive trace.   

Then $\phi$ is a strictly plurisubharmonic function.  Likewise the function $\phi_n: \C^n \to \R$ given by $(z_1, \ldots, z_n) \mapsto \sum \phi(z_i)$ is strictly plurisubharmonic. 
Evidently $\phi_n$ descends to a continuous (not smooth) strictly plurisubharmonic function on $\Sym^n \C$, which we also denote $\phi_n$.

We have the following translation equivariance: 

\begin{lemma} \label{translation equivariance}
        $\phi_{n}(\mathbf{z}) = \phi_{n}(\mathbf{z} - c(\mathbf{z})) + |\mathbf{z}| \cdot  \phi(c(\mathbf{z}))$.
\end{lemma}
\begin{proof}
    By the relation between quadratic and bilinear forms:
    $$\phi(z_i) - \phi(z_i - c(\mathbf{z})) - \phi(c(\mathbf{z})) = \langle z_i - c(\mathbf{z}), c(\mathbf{z}) \rangle_{\phi} $$
    Summing over $i$ kills the RHS and gives the result. 
\end{proof}

That is, $\phi_n$ is the sum of functions on the factors under the splitting
$$ \Sym^n \C \to \C \times \Sym^n_0 \C, \quad \mathbf{z} \mapsto (c(\mathbf{z}), \mathbf{z} - c(\mathbf{z})). $$
Here,  $\Sym^n_0(\C) \subset \Sym^n(\C)$ is the locus of points with center of mass at the origin.

Finally, we argue that $\phi_n$ can be smoothed preserving strict plurisubharmonicity and also certain relevant structural properties:

\begin{proposition}\label{pp: smoothing-KP}
Fix a clustering rule.  Then there exist positive numbers $\epsilon_1 > \epsilon_2 >  \cdots$, and  family of smooth strictly plurisubharmonic functions $\varphi_{n,t}$ on $\Sym^n \C$, $t \in (0, \epsilon_n)$, such that 

\begin{enumerate}
\item \label{normalization and smoothing condition} $\varphi_{1,t} = \phi$, and $(\varphi_{n,t})_{t \in (0, \epsilon_n)}$ is a smoothing of $\phi_n$. 
    \item \label{cluster additivity condition}
    For $\mathbf{z} \in \Sym^n \C$ and a cluster decomposition  $\mathbf{z} = \mathbf{z}_1 + \ldots + \mathbf{z}_m$, and $t \in (0, \epsilon_n)$, we have 
    \begin{equation}\label{cluster additive} \varphi_{n,t}(\mathbf{z}) = \sum_{i=1}^m \varphi_{|\mathbf{z}_i|,t}(\mathbf{z}_i). 
    \end{equation}
    \item  \label{center of mass condition}  
    $$\varphi_{n, t}(\mathbf{z}) = \varphi_{n, t}(\mathbf{z} - c(\mathbf{z})) + |\mathbf{z}| \cdot  \phi(c(\mathbf{z}))$$
\end{enumerate}
\end{proposition}
\begin{proof}

We will inductively construct $\varphi_{n,t}$. The case for $n=1$ is done, we may set $\epsilon_1=1$ for concreteness. Assume the cases for $n<n_0$ are done, and we are considering $n=n_0$.  For any $\mathbf{z}$ with nontrivial cluster decomposition, i.e. $\mathbf{z} \in \mathrm{Clus}^{>(n)}(\C) := \bigcup_{\tau > (n)} \mathrm{Clus}^\tau(\C)$, the expression of 
Condition \eqref{cluster additivity condition} is independent of the choice of nontrivial cluster decomposition by Proposition \ref{cluster decompositions common refinement} and the inductive hypothesis for Condition \eqref{cluster additivity condition}. 
Thus we use this formula to define $\varphi_{n, t}$ on all $\mathbf{z}$ with nontrivial cluster decomposition.  Since the space of elements admitting a cluster decomposition of a given type is open, we may check smoothness of the LHS of \eqref{cluster additive} via the RHS, which is, by induction, smooth for  $t > 0$. 

It remains to define $\varphi_{n, t}$ on the locus $\mathrm{Clus}^{(n)}(\C) \setminus  \mathrm{Clus}^{>(n)}(\C)$ of elements with no nontrivial cluster decomposition. By Condition \ref{center of mass condition}, we only need to define it on $\Clus^{(n)}(\C)_0 \RM \Clus^{>(n)}(\C)_0$.  
$$V = \Clus^{(n)}(\C)_0 \RM \overline{\Clus^{>(n)}(\C)_0},\quad W = \Clus^{>(n)}_{\{r'_i, d'_i\}}(\C)_0 $$
where $\Clus^{\tau}_{\{r'_i, d'_i\}}(\C)_0$ is defined using a more relaxed clustering rules $r'_i > r_i, d'_i < d_i$ (see Lemma \ref{relax r and d} for its existence). If $\bz \in W$ and $\bz$ decomposes (using the more relaxed rule) as $\bz = \sum_i \bz_i$, we set $\varphi_{n,t}^W(\bz) = \sum_i \varphi_{|\bz_i|,t}(\bz_i)$. We also use Richberg theorem \ref{thm:richberg} to have a family of smoothing $\varphi^V_{n,t}$ on $V$. Then we may apply Varouchas Lemma \ref{Varouchas lemma} to patch the two together and get $\varphi_{n,t}$ on $\Clus^{(n)}(\C)_0$, such that on $\Clus^{>(n)}(\C)_0$ it agrees with existing $\varphi_{n,t}$.

We now extend  to all of $\Sym^n(\C)_0$ by enforcing condition \eqref{center of mass condition}.  It remains to check this extension continues to satisfy  conditions \eqref{normalization and smoothing condition} and \eqref{cluster additivity condition}.  Condition \eqref{normalization and smoothing condition} follows from 
Lemma \ref{translation equivariance}.  

For condition \eqref{cluster additivity condition}, we must check that the expression \eqref{cluster additive} inductively ensures condition \eqref{center of mass condition} on $\bigcup_{\tau > (n)} \mathrm{Clus}^\tau(\C)$. 
Suppose given a cluster decomposition $\mathbf{z}=\mathbf{z}_1+\cdots+\mathbf{z}_m$.  Then:

\begin{eqnarray*} \varphi_{n,t}(\mathbf{z}) - \varphi_{n, t}(\mathbf{z} - c(\mathbf{z})) & \stackrel{(2)}{=} & \sum \varphi_{|\mathbf{z}_i|, t} (\mathbf{z}_i) - \varphi_{|\mathbf{z}_i|, t} (\mathbf{z}_i - c(\mathbf{z})) \\
& \stackrel{(3)}{=} & 
\sum \varphi_{|\mathbf{z}_i|, t} (\mathbf{z}_i - c(\mathbf{z}_i)) + |\mathbf{z_i}| \phi(c(\mathbf{z}_i)) -  
\varphi_{|\mathbf{z}_i|, t} (\mathbf{z}_i - c(\mathbf{z}_i)) - |\mathbf{z_i}| \phi(c(\mathbf{z}_i) - c(\mathbf{z}))
\\
& = & \sum |\mathbf{z_i}|(
\phi(c(\mathbf{z}_i)) - \phi(c(\mathbf{z}_i)- c(\mathbf{z}))) \\
& = & \sum |\mathbf{z}_i| (\langle c(\mathbf{z_i}) - c(\mathbf{z}), c(\mathbf{z}) \rangle + \phi(c(\mathbf{z})) \\ 
& = & |\mathbf{z}| \phi(c(\mathbf{z}))
\end{eqnarray*}
This completes the proof.
\end{proof}

We will now axiomatize  relevant properties of $\varphi$. 
Let $\Re: \Sym^n \C \to \Sym^n \R$ be the coordinate-wise real part.  We note: 

\begin{lemma} \label{re inverse can induce splitting}
    Fix a clustering rule,  which we use  both for $\R$ and $\C$. Then 
    $\Re^{-1}(\Clus^{ \tau}(\R)) \subset  \Clus^{\ge \tau}(\C)$.
\end{lemma}
\begin{proof} 
Apply Corollary \ref{contraction induces merging} to the contraction $\Re$.
\end{proof}

Recall that for any ordered  partition $\vec \tau = (\tau_1, \ldots, \tau_\ell)$ of $n$, there is a natural embedding $\Sym^{\vec \tau}(\R) \hookrightarrow{\R^{\ell}}$ taking  $\sum \tau_i z_i$ (with $z_1 < z_2 < \cdots < z_{\ell}$) to $(z_1, z_2, \ldots,  z_{\ell})$. 

Consider the diagonal embedding 
\begin{eqnarray*}
   i_{\vec{\tau}}: \R^\ell & \to & \R^n \\
   (z_1, \ldots, z_\ell) & \mapsto & (\underbrace{z_1, \ldots z_1}_{\tau_1},  \ldots, \underbrace{z_\ell, \ldots, z_\ell}_{\tau_\ell})
\end{eqnarray*}
We give $\R^\ell$ a metric by pulling back the standard metric on $\R^n$ (which gives  $\sum_{i=1}^\ell \tau_i (d z_i )^2$), and use this to identify $T \R^\ell \cong T^* \R^\ell$. 
\begin{equation} \label{cotangent embedding re inverse}
    \Re^{-1}(\Sym^{\vec \tau}(\R)) \hookrightarrow \C^\ell \cong T \R^{\ell } \cong T^* \R^{\ell}
\end{equation}

We fix a clustering rule and take the preimages of the identification  \eqref{product structure clus R}: 
\begin{align} \notag \Re^{-1}\Clus^{\vec{\tau}}(\R)  & \xrightarrow{\sim}   (\Clus^{\vec{\tau}}(\R) \cap \Sym^{\vec{\tau}}(\R) ) \times \left( \Re^{-1} (\Clus^{(\tau_1)}(\R)_0) \times \cdots \times \Re^{-1} (\Clus^{(\tau_\ell)}(\R)_0) \right) \\
\label{re inverse product}
 = & \,\, \Re^{-1}(\Clus^{\vec{\tau}}(\R) \cap \Sym^{\vec{\tau}}(\R) ) \times \left( (\Re^{-1} \Clus^{(\tau_1)}(\R))_0 \times \cdots \times (\Re^{-1} \Clus^{(\tau_\ell)}(\R))_0 \right) &
\end{align}
(In the first line, the second factor consists of configurations whose center of mass has real part zero; in the line, the same terms are now asked to center of mass zero, and the imaginary part of the center of mass is now moved into the first factor.)

\begin{definition} \label{cm standard 1-form}
    We say a strictly plurisubharmonic function $\varphi: \Sym^n \C \to \R$ is {\em cm standard} if $\varphi|_{\Re^{-1}\Clus^{\vec{\tau}}(\R)}$ is a sum $\varphi_{cm} + \varphi_{int}$ 
    of functions on the two factors of the product on the RHS of \eqref{re inverse product}, and moreover $$\varphi_{cm}(z_1, \ldots, z_\ell) = \sum_{i = 1}^\ell \tau_i \mathrm{Im}(z_i)^2.$$

    We say a 1-form on $\lambda$ on $\Sym^n \C$ is {\em cm standard} if, for any partition $\tau$, the restriction $\lambda|_{\Re^{-1}\Clus^{\vec{\tau}}(\R)}$ is a sum $\lambda_{cm} + \lambda_{int}$ of 1-forms pulled back from the  factors, and moreover $\lambda_{cm}$ is pulled back from the standard cotangent form via \eqref{cotangent embedding re inverse}.
    It is evident that if $\varphi$ is a cm-standard   strictly plurisubharmonic function, then $d^c \varphi$ is a cm standard 1-form. 
\end{definition}

Assuming $\varphi_{int}$ is controlled at infinity, 
it follows from the required form of $\varphi_{cm}$ that such $\varphi$ is proper, hence determines a Stein structure.

\begin{corollary}
 The    $\varphi_{n, t}$ of Proposition \ref{pp: smoothing-KP} is a cm-standard plurisubharmonic function on $\Sym^n \C$. 
\end{corollary}
\begin{proof}
By definition of cm-standard psh function, we need to show for any ordered partition $\vec \tau$, $$\varphi|_{\Re^{-1}(\Clus^{\vec \tau})(\R)}(\bz) = \varphi_{cm,\tau}(\bz) + \varphi_{int,\tau}(\bz).$$ 
Since $\Re^{-1}(\Clus^{\vec \tau}) \In \cup_{\sigma \geq \tau} \Clus^\sigma(\C)$, we may assume $\bz$ is contained in some refinement $\sigma$ of the (unordered) partition $\tau$. Assume $\bz = \sum_{i=1}^l \bz_i = \sum_{i,j} \bz_{i,j}$, where $Re(\bz) = \sum_{i} \Re(\bz_i)$ is the ordered partition of $\tau$, and the $\bz_{i,j}$ is the $\sigma$ refinement. By Proposition \ref{pp: smoothing-KP} (2) and (3), with $\phi(z) = \Im(z)^2$, we know
$$ \varphi|_{\Clus^\sigma(\C)} = \sum_{i,j} |\bz_{i,j}| |\Im(c(\bz_{i,j}))|^2 + \sum_{i,j} \varphi_{| \bz_{i,j}|,t}(\bz_{i,j} - c(\bz_{i,j})) $$
For each $i$, the sum over $j$ can be simplified using Lemma \ref{translation equivariance}, 
$$ \sum_{j} |\bz_{i,j}| |\Im(c(\bz_{i,j}))|^2 = |\bz_{i}| |\Im(c(\bz_i))|^2 + \sum_{j} |\bz_{i,j}| |\Im(c(\bz_{i,j}) - c(\bz_{i}))|^2. $$
We then get
$$ \varphi|_{\Clus^\sigma(\C)} = \sum_{i} |\bz_{i}| |\Im(c(\bz_i))|^2 + \sum_i \varphi_{int,\tau}(\bz_{i,j} - c(\bz_i)). $$
The first sum is $\varphi_{cm,\tau}$, and the remainder depends only on position relative to cluster centers. 
\end{proof}

\begin{proposition} \label{pp:smoothing-KP-general}
    Let $\Sigma$ be a Riemann surface with K\"ahler potential   $\phi$. Assume the Riemannian metric $g_\phi$ has a lower bound of convexity radius $r_{conv}$. Fix $N>0$, and fix a clustering rule $r_i, d_i$ for $i=1,\cdots, N$, such that $r_N, d_N \ll r_{conv}$. Let $I = T^*[-2,2] \cong \Re^{-1}([-2,2])$, with K\"ahler potential $\phi_I(z) = \Im(z)^2$. Assume there is a closed embedding $\iota_A: A := I^{\sqcup n_I} \into \Sigma$ that preserves K\"ahler potential. 
    
     Then, there exists $\epsilon_1 > \epsilon_2 >\cdots > \epsilon_N$, and smooth K\"ahler potential $\varphi_{n,t} $ on $\Sym^n(\Sigma)$, for $0 < n \leq N$ and $t \in (0, \epsilon_n)$, such that 
    \begin{enumerate}
        \item $\varphi_{1,t} = \phi$, and $\sup_{\bz}|\varphi_{n,t}(\bz) - \phi_n(\bz)| < t$ for all $t \in (0, \epsilon_n)$. 
        \item \label{general Sigma, KP factorize} If $\bz$ admits a cluster decomposition $\bz = \bz_1 + \cdots + \bz_m$, then 
        $$ \varphi_{|\bz|,t}(\bz) = \sum_{i=1}^m \varphi_{|\bz_i|,t}(\bz_i). $$
        \item If $\bz \in \Sym^n (T^*[-2+r_n,2-r_n])$, where $T^*[-2,2]$ is one component of $A$, then $\varphi_{n,t}(\bz)$ agrees with the K\"ahler potential in Prop \ref{pp: smoothing-KP}
    \end{enumerate}
\end{proposition}
\begin{proof}
    We construct $\varphi_{n,t}$ by induction. Assume for $n < n_0$, we have $\varphi_{n,t}$ constructed that satisfies the above conditions.  The case $n=1$ is given. 
    For $n=n_0$, we only need to construct $\varphi_{n,t}$ on $\Clus^{(n)}(\Sigma)$, that is compatible with $\Clus^{>(n)}(\Sigma)$ and $\Clus^{(n)}(A)$. Let 
    $$W = \Clus^{>(n)}_{\{r'_i, d'_i\}}(\Sigma) \cup \Clus^{(n)}(Int(A)), \quad  V = \Clus^{(n)}(\Sigma) \RM ( \overline{(\Clus^{>(n)} (\Sigma) \cup \Sym^n(\iota_A(A'_n)}) $$
    where $r'_i, d'_i$ is a slight relaxed clustering rule Lemma \ref{relax r and d}
    , and $A'_n = (T^*[-2+r_n, 2-r_n])^{n_I} \In A$ is a slight shrinking of $A$. We choose $\varphi_{n,t}^V$ using Richberg's Theorem (\ref{thm:richberg}), and $\varphi^W_{n,t}(\bz)$ is either determined apply cluster decomposition of $\bz$ using the relaxed clustering rule, or by construction in Prop \ref{pp: smoothing-KP} for $\bz \in \Clus^{(n)}(Int(A))$. Then, we may apply Varouchas's Lemma (\ref{Varouchas lemma}) to patch the two parts. 
\end{proof}

We again axiomatize the relevant properties of the above construction. 
Let $\Sigma = \Sigma_L \cup_A \Sigma_R$ be a gluing configuration.  We fix an identification $A = \coprod_\alpha T^* [-2,2]$.  We will consider the projection $\pi: \Sigma \to [-2,2]$ which sends $\Sigma_L \mapsto -2$, $\Sigma_R \mapsto 2$, and sends $A \to [-2,2]$ as the projection of the cotangent bundles to the base.

For any $(b, c) \subset [a,d] \subset \ (-2,2)$, with $|b-a|>2r_n$ and $|d-c|>2r_n$, we have an embedding 
\begin{equation} \label{edge versus complement cover}
    \Omega_{l,k,r} := \Sym^{l}(\Sigma_L \setminus T^*[a,d]) \times 
    \Sym^k(\coprod_\alpha T^*(b,c)) \times \Sym^r(\Sigma_R \setminus T^*[a,d]) \hookrightarrow \Sym^{l+k+r}(\Sigma)
\end{equation}
Note that these sets, as $a, b,c,d$ vary, and $l,k,r$ varies with $n=l+k+r$, cover $\Sym^n \Sigma$. 
\begin{definition} \label{def: adapted}
    We say that a plurisubharmonic function $\phi: \Sym^n \Sigma \to \R$ is adapted to the gluing configuration if, on every open of the form \eqref{edge versus complement cover}, it is the sum of pullbacks of plurisubharmonic functions, and the factors on the $T^*(b,c)$ are functions constructed in Proposition \ref{pp: smoothing-KP} from some fixed original $\phi:T^*(b,c) \to \R$.  
    
    We say  a 1-form $\lambda$ on $\Sym^n \Sigma$ is adapted to the gluing configuration $\Sigma = \Sigma_L \cup_A \Sigma_R$ if, on every open of the form \eqref{edge versus complement cover}, $\lambda$ splits as a direct sum, and the factors on the $T^*(b,c)$ are cm-standard forms.
    Evidently $d^c$ of an adapted plurisubharmonic function is an adapted form. 
\end{definition}

\begin{lemma}
    If the $A: \coprod T^*[-2,2] \to \Sigma$ of  Proposition \ref{pp:smoothing-KP-general} arises from a gluing configuration, then the 
    plurisubharmonic functions $\varphi_{n,t}$ are adapted to the gluing configuration.  
\end{lemma}
\begin{proof}
We have $\Sigma_L \setminus T^*[a,d] \sqcup T^*(b,c) \sqcup \Sigma_R \setminus T^*[a,d] \In \Sigma$, with boundary gaps with $|b-a|>2r_n$ and $|d-c|>2r_n$, hence points from different parts cannot belong to the same cluster. Then the factorizability condition of $\varphi_{n,t}$ for cluster decomposition (Prop \ref{pp:smoothing-KP-general}, condition \ref{general Sigma, KP factorize}) gives the desired product structure of $\varphi_{n,t}$. 
\end{proof}

\section{Sectorial covers for symmetric products} 

In this section we construct sectorial covers of symmetric products, by pulling back certain covers of $\Sym^n \R$.  We then prove Theorems \ref{thm: functor} and \ref{thm: intro cover}.

\subsection{Box decomposition of $\Sym^n \R$}

Fix $b > 0$.  We write 
$$ B_k(x) = [x - k b, x+kb] \subset \R$$ 
For $\mathbf{z} \in \Sym^n \R$, we take 
$$ B_{\mathbf{z}} = B_{|\mathbf{z}|}(c(\mathbf{z})). $$

We say $\bz$ is {\bf well contained} in a box $B_\bz$ if for any consecutive subset of points $\bz' \In \bz$ we have $B_{\bz'} \In B_{\bz}$.

\begin{lemma}\label{lm:no-splitting}
    If $\bz$ is well contained in $B_\bz$, then for any decomposition $\bz = \bz_1 + \bz_2 + \cdots + \bz_m$, there exists some overlap $B_{\bz_i} \cap B_{\bz_{j}}$ . 
\end{lemma}
\begin{proof}
    Otherwise,  we have $\sqcup_i B_{\bz_i} \subsetneq B_{\bz}$, which implies a contradiction:
    $$ |\bz|=len(B_{\bz}) > \sum_i len(B_{\bz_i}) = \sum_i |\bz_i| = |\bz|.$$
\end{proof}

\begin{lemma}\label{lm:cluster-in-box}
For any $\mathbf{z} \in \Sym^k \R$, there is a unique decomposition $\mathbf{z} = \mathbf{z}_1 + \cdots + \mathbf{z}_m$, such that each $\bz$ is well-contained in $B_{\bz}$, and the $B_{\mathbf{z}_i}$ are disjoint and ordered $B_{\mathbf{z}_1} < B_{\mathbf{z}_2} < \cdots < B_{\mathbf{z}_m}$. 
\end{lemma}
\begin{proof}
For existence, we begin with the decomposition into one-point parts, $\mathbf{z} = z_1 + \ldots + z_k$.  Each fits into its own  box, but the boxes may overlap. If two boxes overlap, we merge the corresponding summands. It is easy to check that the new collection of points are well-contained in its box. We repeat the process until no pair of boxes overlap anymore.  

For uniqueness, if there are two such decompositions of $\bz = \sum_{i=1}^m \bz_i = \sum_{j=1}^l \bz'_j$, then (up to exchange $\bz'_\bullet$ with $\bz_\bullet$) there exists $\bz_i = \sum_{j=1}^l \bz_i \cap \bz'_j$ with more than one non-empty summand. We have $B_{\bz_i \cap \bz'_j} \In B_{\bz'_j}$ for all $j$ by 'well-contained' definition, and hence $B_{\bz_i \cap \bz'_j}$ are all disjoint for all $j$. However, this forms a non-trivial splitting of $\bz_i$, which contradicts $\bz_i$ is well-contained in $B_{\bz_i}$ by Lemma \ref{lm:no-splitting}. 
\end{proof}

We now discuss an open cover of $\Sym^n \R$, which we view as associated to the cover of $\R$ by $U_L = (-\infty,1)$ and $U_R = (-1, \infty)$. 

\begin{definition} \label{def: box opens}
Fix $\ba=(a_0, \cdots, a_n)$ with
$$ -1 \leq a_0 < a_1 < \cdots < a_n \leq 1 $$
such that $a_{i+1}-a_i > 2b$. 
For any partition $n= i + (n-i)$, $0 \leq i \leq n$, we say $\bz \in \Sym^n \R$ is in $U^{\ba}_{i, n - i}$ if $a_{i}$ 
is in the complement of the union of boxes for the box decomposition of $\bz$, and the total number of points to the left of $a_i$ is $i$. In terms of boundary defining equations, we set $$ \rho^{\ba}_{k,n-k}(\bx) = \max( \{ x_{[k-i,k]} + (i+1)b - a_k\}_{i=0}^k, \{a_k - x_{[k+1,k+1+i]} + (i+1)b\}_{i=0}^{n-k}), $$ then $$ U^{\ba}_{k,n-k}=\{\rho^{\ba}_{k,n-k} < 0\}. $$
To have a smoothed version, we set $$ \rho^{\ba, \delta}_{k,n-k} := M_\delta( \{ x_{[k-i,k]} + (i+1)b - a_k\}_{i=0}^k, \{a_k - x_{[k+1,k+1+i]} + (i+1)b\}_{i=0}^{n-k}) $$
        and set 
        $$ U^{\ba, \delta}_{k,n-k} := \{\rho^{\ba, \delta}_{k,n-k} < 0 \} $$
\end{definition}

\begin{lemma} \label{box opens cover}
    
    The collection $U^\ba_{n,0}, U^\ba_{n-1,1}, \cdots, U^\ba_{0,n}$ forms an open cover of $\Sym^n \R$. Similarly the smoothed version $\{U^{\ba,\delta}_{k,n-k}\}$. 
\end{lemma}
\begin{proof}
Openness is immediate because the boxes are closed.  We now argue that these opens cover.

Let $\bz \in \Sym^n \R$, with box decomposition $\bz = \bz_1 + \cdots + \bz_m$.  Let $B = \cup B_{\bz_i}$ and
$$ g(x) = \int_{-\infty}^x \frac{1}{2b} 1_B(t) dt. $$

Let $f(x)$ be a smooth monotone increasing function, such that $f(a_i)=i$,  $0< f'(x) < 1/2b$ (this is possible since the gap of $a_i$'s is bigger than $2b$), and $f'(x)=1$ for $|x| \gg 1$. 

Let $x_1< \cdots< x_N$ denote the $x$-coordinate of the graph intersections of $f(x)$ and $g(x)$.  We label the intersections by $+, -, 0$: 
\begin{itemize}
    \item $0$ for $g'(x_i)$ not defined, i.e. $x_i \in \d B$
    \item  $+$ for $g'(x_i) = 1/2b > f'(x_i)$, i.e  $x_i \in int(B)$
    \item $-$ for $g'(x_i) =0 < f'(x_i)$, i.e. $x_i \notin B$
\end{itemize}
Since for $|x| \gg 1$, $f'(x) = 1 > 0 = g'(x)$, there will be exactly one more $-$ labelled point than $+$, in particular there is at least one $-$ labelled point. For any intersection point $x$ with $-$ label, we have $g'(x)=0$ and $g(x) = i \in \{0, \cdots, n\}$, hence $f(x)=i$, thus $x=a_i$. This shows $g(a_i) = i$, i.e. $\bz \in U^\ba_{i,n-i}$. 
\end{proof}

When (as will usually be the case) the precise choice of $\ba$ does not matter, we may omit it and simply write $U_{i,j}$.

\begin{definition} \label{def:refined-cover}
    For any ordered partition $n = n_L + m_1 + \ldots + m_k + n_R$    we say $\bz \in \Sym^n \R$ is in $U^{\ba}_{n_L, (m_i)_i, n_R}$ if $a_{i}$ 
is in the complement of the union of boxes for the box decomposition of $\bz$, and the total number of points to the left of $a_{n_L}$ is $n_L$, and to the left of $a_{n_L + m_1 + \ldots + m_k}$ is $n_L + m_1 + \ldots + m_k$.  
Or equivalently,
$$ U^\ba_{n_L, (m_i)_i, n_R} := U_{n_L, m_1 + \cdots + n_R}^\ba \cap U^\ba_{n_L+m_1, m_2 + \cdots + n_R} \cap \cdots \cap U^\ba_{n_L + \cdots + m_k, n_R}.$$ 
\end{definition}

\begin{lemma} \label{U inclusions}
    We have $U^{\ba}_{n_L', (m_i')_i, n_R'} \subset U^{\ba}_{n_L, (m_i)_i, n_R}$ if and only if the ordered partition $n = n'_L + \cdots + n'_R$ is a refinement of the ordered partition $n = n_L + \cdots + n_R $
\end{lemma}
\begin{proof}
    Immediate from the definition. 
\end{proof}

\begin{lemma} \label{U intersections}
    All intersections of the $U^\ba_{i,j}$ are of the form $U^\ba_{n_L, (m_i)_i, n_R}$. 
\end{lemma} 
\begin{proof}
    Consider some collection $U_{n_L^j, n_R^j}$ where $n_L^j + n_R^j = n$, are distinct decompositions of $n$ for $j=0,\ldots,k$.  (So $k \le n$.)  We order so that $n_L^0 < \cdots < n_L^k$.
    The point is just that we can re-express the $n_j$ in terms of appropriately chosen $m_i$.  For $j=1,\cdots,k$, let $m_j = n_L^j - n_L^{j-1}$, and let $n_L = n_L^0$, $n_R = n_R^k$. Then by definition 
    $$U_{n_L^0, n_R^0}^{\mathbf{a}} \cap \cdots \cap U_{n_L^k, n_R^k}^{\mathbf{a}}  =  U_{n_L, (m_i)_i, n_R}^{\mathbf{a}} .$$
\end{proof}

\begin{lemma} \label{U product}
    Given $n = n_L + m_1 + \ldots + m_k + n_R$, and fixed $\mathbf{a}$, we let $\mathbf{a}_L := (a_0, \ldots, a_{n_L})$, 
    $\ba_{m_1} := (a_{n_L}, \ldots, a_{n_L + m_1})$,
    $\ba_{m_i} := (a_{n_L + \ldots + m_{i-1}}, \ldots, a_{n_L + \ldots + m_{i}})$, 
    and $\ba_{n_R} := (a_{n_L+ m_1 + \cdots + m_k}, \ldots, a_n)$.  Then
    
    $$U^\ba_{n_L, (m_i)_i, n_R}= U^{\ba_L}_{n_L, 0} \times \cdots \times U_{0, m_i, 0}^{\ba_{m_i}} \times \cdots \times U^{\ba_R}_{0, n_R}$$ 
\end{lemma}
\begin{proof}
    Immediate from the definition. 
\end{proof}

In the remainder of the subsection, we establish a crucial property of the $U_{i,j}$: their boundaries are defined by center of mass functions.
To show this we will have to compare cluster and box decompositions.

First note any element $\mathbf{z} \in \Sym^n (\R)$ has a canonical decomposition $\mathbf{z} = \sum \tau_i z_i$ where $z_1 < z_2 < \ldots \in \R$ and $\tau_i$ are the multiplicities.  We regard $\vec\tau := (\tau_1, \ldots, \tau_l)$ as an ordered partition (and write $[\vec \tau]$ for the corresponding unordered partition), and write $\Sym^{\vec \tau}(\R) \subset \Sym^n(\R)$ for the locus of such points.  Evidently  
$$\Sym^{\tau}(\R) = \coprod_{[\vec \tau] = \tau} \Sym^{\vec{\tau}}(\R)$$
and we have a canonical identification $\Sym^{\vec \tau}(\R) \cong \R^{\ell(\vec \tau)}$, where $\ell(\vec \tau)$ is the number of parts of the partition.

We similarly write  $\mathrm{Clus}^{\vec \tau}(\R)$ for the locus of points which admit an ordered cluster decomposition $\mathbf{z} = \mathbf{z}_1 + \ldots + \mathbf{z}_l$, where $|\mathbf{z}_i| = \tau_i$ and, if $i < j$, then $c(\mathbf{z})_i < c(\mathbf{z}_j)$, or equivalently, all points in $\mathbf{z}_i$ are less than all in $\mathbf{z}_j$.  Note  there is a not-disjoint-union  
\begin{equation}\label{clus R to sym R}\Clus^{\tau}(\R) = \bigcup_{[\vec \tau] = \tau} \Clus^{\vec{\tau}}(\R)
\end{equation}

We write $\Sym^n (\R)_0$ for the subset of configurations with center of mass at zero, and similarly for $U \subset \Sym^n(\R)$ we write $U_0:= U \cap \Sym^n(\R)_0$.

\begin{lemma} 
For $\vec{\tau} = (\tau_1, \ldots, \tau_\ell)$,
there is a diffeomorphism 
\begin{eqnarray}
\label{product structure clus R}\Clus^{\vec{\tau}}(\R) & \xrightarrow{\sim} & (\Clus^{\vec{\tau}}(\R) \cap \Sym^{\vec{\tau}}(\R) ) \times \Clus^{(\tau_1)}(\R)_0 \times \cdots \times \Clus^{(\tau_\ell)}(\R)_0 \\
\notag \mathbf{z}_1 + \cdots + \mathbf{z_{\ell}} & \mapsto & ((|\mathbf{z}_1|\cdot c(\mathbf{z}_1), \ldots, |\mathbf{z}_\ell|\cdot  c(\mathbf{z}_\ell)), \mathbf{z}_1 - c(\mathbf{z}_1), \ldots, \mathbf{z}_\ell - c(\mathbf{z}_\ell))
\end{eqnarray}
We will denote the composition of this map with the projection to the first factor by $c_{\vec \tau}$. 
\end{lemma}
\begin{proof}
    The formula evidently defines an injective map to $\Sym^{\vec{\tau}}(\R) \times \Sym^{\tau_1}(\R)_0 \times \cdot \times \Sym^{\tau_\ell}(\R)_0$.  The condition that the first factor lie in $\Clus^{\vec{\tau}}(\R)$ is tautologically equivalent to the separation condition on the original decomposition $\mathbf{z}_1 + \ldots + \mathbf{z}_k$, and the condition that the remaining factors land in the $\Clus^{(\tau_k)}(\R)$ is the confinement condition.
\end{proof}

Evidently $c_{\vec \tau}$ is the restriction of $c_\tau$, so 
a function $f: \Sym^n \R \to X$ is a center of mass function iff $f|_{\mathrm{Clus}^{\vec\tau}(\R)}$ factors through $c_{\vec \tau}$.

\begin{proposition} \label{box opens have cm boundary}
    Fix $\ba=(-1\leq a_0<\cdots<a_n \leq 1)$. Fix a clustering rule, and assume that $b \gg r_i$ for $i \le n$, and fix $\delta \ll b$.  
    \begin{enumerate}
        \item  $\rho^{\ba}_{k,n-k}$ is a cm function near $\d U^{\ba}_{k,n-k}$. 
        \item  $\rho^{\ba, \delta}_{k,n-k}$ is a cm function near $\d U^{\ba, \delta}_{k,n-k}$. 
    \end{enumerate}
    
\end{proposition}
\begin{proof}
$U^{\ba}_{k,n-k}$'s boundary consists of two types of generic behavior, either a size $i+1$ box hits the boundary point $a_k$ from the left, which is given by $x_{[k-i,k]} + (i+1)b = a_k$; or a size $i+1$ box hits the boundary point $a_k$ from the right, which is given by $a_k = x_{[k+1,k+1+i]} - (i+1)b$. In either case, the conditions are expressed using box centers, and box centers are cm functions. Since $b \gg r_n$, the kinks of $\rho^{\ba}_{k,n-k}$ is well inside the cluster opens $\Clus^{\vec \tau}(\R)$, hence the smoothing $\rho^{\ba, \delta}_{k,n-k}$ will remain a cm function. 
\end{proof}

\begin{proposition} \label{prop: transverse boundaries}
    The boundaries of $U^{\ba, \delta}_{k,n-k}$ of different $k$ meet transversely.
\end{proposition}
\begin{proof}

Fix $0 \leq k_1 <\cdots < k_r \leq n$. Consider the continuous map
$$ \Pi:=\prod_{i=1}^r \rho^{\ba}_{k_i, n-k_i} : \Sym^n \R \to \R^r.$$
For $i=0, \cdots, r$, set $m_i = k_{i+1}-k_i$. Then we have
$$ \Pi^{-1}( (-\infty, 0]^r) = \bigcap_{i=1}^r \overline{U^\ba_{k_i, n-k_i} }= \wb{U^{\ba}_{m_0, m_1, \cdots, m_r}}.$$
And in particular, we have
$$ \Pi^{-1}(0) = \bigcap_{i=1}^r \d U^\ba_{k_i, n-k_i} \In \overline{U^{\ba}_{m_0, \cdots, m_r}}. $$
We can describe the points in $\Pi^{-1}(0)$ explicitly. 

First, note that $\bz \in \d U^\ba_{k, n-k}$ if and only if
\begin{enumerate}
    \item $\bz = \bz_0 + \bz_1$, where $\bz_0 \in \Sym^k (-\infty, a_k)$ and $\bz_1 \in \Sym^{n-k}(a_k, \infty)$
    \item  Boxes covering $\bz_0$ (as points in $\R$) should be contained in $(-\infty, a_k]$;  boxes covering $\bz_1$ (as points in $\R$) should be contained in $[a_k, +\infty)$. 
    \item $a_k$ is covered by either boxes covering $\bz_0$ or boxes covering $\bz_1$. 
\end{enumerate}

Let $I_0, \cdots, I_r$ denote the connected components (from left to right) of $\R \RM \{a_{k_1}, \cdots, a_{k_r}\}$. Then $\bz \in \Sym^n \R$ belong to $\Pi^{-1}(0)$, if and only if the following conditions are met
\begin{enumerate}
    \item $\bz = \bz_0 + \cdots + \bz_r$, where $\bz_i = \bz \cap I_i$, $|\bz_i|=m_i$. 
    \item boxes covering $\bz_i$ should be contained in $\overline{I_i}$
    \item each $a_{k_i}$ should covered by some boxes' boundaries in either $\bz_{i-1}$ or $\bz_i$. 
\end{enumerate}

Since $a_{k_{i+1}} - a_{k_i} > (k_{i+1} - k_i) \cdot 2b = |\bz_{i+1}| \cdot 2b$, there is no box in interval $I_{i+1}$ that touches both endpoints. 

If we write $\bz = (x_1 \leq x_2 \cdots \leq x_n)$, and suppose the boxes that touch $a_{k_i}$ contains points $x_{s_i}, x_{s_i+1}, \cdots, x_{t_i}$. That means the boundary defining function for $U^\ba_{k_i, n-k_i}$ in a small neighborhood of $\bz$ only involves coordinates $x_{s_i}, x_{s_i+1}, \cdots, x_{t_i}$, which is different from the boundary defining function for $U^\ba_{k_j, n-k_j}$, if $j \neq i$. This shows the $r$ exterior conormal cones of $d\rho^{\ba}_{k_i, n-k_i}|_\bz$ are linearly independent. 

Now we can consider the smoothed defining functions $\rho_i^{sm}:=\rho^{\ba, \delta}_{k_i, n-k_i}$, where $\delta \ll b$. By similar argument above, if $\bz \in \cap_{i=1}^r \d U^{\ba,\delta}_{k_i, n-k_i}$, then the $x_s$ coordinates $\rho_i^{sm}$ locally uses will be different from those used by $\rho_j^{sm}$, for $i\neq j$. This shows their differentials $d \rho_i^{sm}$ are linearly independent. 

\end{proof}

\subsection{Recollections on Liouville structures}
\label{sec: sector review}

Given an exact symplectic manifold $(X, \omega = d \lambda)$, the {\em Liouville vector field} $Z$ is characterized by $\omega(Z, \cdot) = \lambda$.  Then $X$ is said to be {\em Liouville} if there is a compact submanifold-with-boundary $\overline{X} \subset X$ such that $Z$ is outwardly transverse to $\partial \overline{X}$ and the flow of $Z$ identifies $X \setminus \overline{X}$ with a half-cylinder on $\partial \overline{X}$.  In this case $\lambda|_{\partial \overline{X}}$ is a contact form, and the corresponding contact manifold is canonically independent of the choice of $\overline{X}$; we correspondingly denote it $\partial_\infty X$ and term it the {\em ideal contact boundary} of $X$.  The submanifold $\overline{X}$ is termed a {\em Liouville domain} with completion $X$.  

We now recall from \cite{GPS1, GPS2} the definitions and some properties of Liouville sectors.

\begin{definition} \cite{GPS1}
    Let $(X, \lambda)$ be a Liouville manifold.
    A closed codimension one smooth submanifold $H \subset X$ is a {\em sectorial hypersurface} if there is a function $I: H \to \R$ such that (1) $ZI = I$ in the complement of a compact set and (2) $dI$ is nonvanishing on the characteristic foliation of $H$. 
\end{definition}

    Note that if $H = \{\rho = 0\}$ with $d\rho|_H$ non-vanishing, then  (2) above is equivalent to asking the Poisson bracket $\{\rho, I\}$ to be nowhere zero. 

\begin{definition} \cite[Def. 12.2]{GPS2}
     A union of cleanly intersecting sectorial hypersurfaces $H_i$ with characteristic foliations $\xi_i$ is {\em sectorial} if (the functions $I_i$ admit extensions to neighborhoods of the $H_i$ so that) 
    the $\xi_i$ are symplectically orthogonal over the intersections of the $H_i$, the $I_i$ are in involution ($\{I_i, I_j\} = 0$), and $d I_i|_{\xi_j} = 0$ when $i \ne j$. 
\end{definition}

A model and motivating example is that if $M = \bigcup U_i$ is a cover of a smooth manifold by open sets with smooth transverse boundaries, with boundary-defining-functions $\rho_i$, then the $T^* U_i$ give a sectorial cover of $T^*M$, with witness functions given e.g. by choosing a metric and taking $I_i$ to be the function on the cotangent bundle given by evaluation of a covector at $ \nabla \rho_i$.

\subsection{Sectorial covers}

\begin{lemma} \label{lm:sectorial-base-condition}
Let $\lambda$ be a cm-standard Liouville form on $\Sym^n \C$, and let $\rho: \Sym^n \R \to \R$ be a smooth center of mass function, with regular value $t$.  Then $\Re^{-1}(\rho^{-1}(t))$ is a sectorial hypersurface.
\end{lemma}
\begin{proof}
By hypothesis, $\rho|_{\Clus^{\vec\tau}(\R)}$ is pulled back from a function $\rho_{\vec \tau}$ on $\Sym^{\vec \tau}(\R) \subset \R^{\ell(\tau)}$.  We write $\nabla \rho_{\vec \tau}$ for the gradient with respect to the metric on $\R^{\ell(\tau)}$ pulled back by diagonal embedding, and define
$I: T^* \Sym^{\vec \tau}(\R) \to \R$ by 
$I(x, p) = p(\nabla \rho_{\vec \tau})$, and define $I$ on $\Re^{-1}(\Clus^{\vec \tau}(\R))$ by pulling back under \eqref{re inverse product}.  We claim that the resulting functions agree on overlaps and hence define a function $I: \Sym^n \C \to \R$.   Indeed, if we lift $\rho$ to $\wt \rho: \R^n \to \R$, and use Euclidean metric on $\R^n$, we may define $\wt I = p \cdot \nabla_{\wt \rho}$ on $T^*\R^n \cong \C^n$. It is easy to check that $\wt I: \C^n \to \R$ descend to $I$.

It remains to check that $I$ is linear at infinity, and $\{I, \rho\}$ is nowhere vanishing along $\{\rho = 0\}$.  
These we may check in charts, where we have $\lambda = \lambda_{cm} \oplus \lambda_{int}$, hence also $\omega = \omega_{cm} \oplus \omega_{int}$ and Liouville vector field given by $Z_{cm} \oplus Z_{int}$.  Thus the verifications descend to the corresponding assertions on $T^* \R^{\ell(\tau)}$ 
\end{proof}

The following generalization of Lemma \ref{lm:sectorial-base-condition} has a similar proof: 

\begin{lemma} \label{lem: sectorial collection criterion}
If $\rho_1, \ldots, \rho_r$ are center of mass functions and $(\rho_1, \ldots, \rho_r): \Sym^n \R \to \R^r$ is regular over $(t_1, \ldots, t_r) \in \R^r$, then the hypersurfaces $\Re^{-1} \rho_i^{-1}(t_i)$ form a sectorial collection.  
\end{lemma}
\begin{proof}
 By shifting $\rho_i$ we may assume all $t_i=0$. Let $\wt \rho_i: \R^n \to \R$ be the lift of $\rho_i$, then $\Pi:=(\wt \rho_1, \cdots, \wt \rho_r): \R^n \to \R^r$ is $S_n$-invariant  with $0$ a regular value. Let $\d_{t_i}$ be coordinate vector fields on $\R^r$. We claim near $\Pi^{-1}(0)$, we can lift $\d_{t_i}$ to vector fields $X_i$ such that $[X_i, X_j]=0$. Let $F = \Pi^{-1}(0)$. Consider the orthogonal splitting $T^*\R^n|_F = TF \oplus NF $, where $NF$ is the normal bundle. Using exponential map, we may identify a tubular neighborhood $U(F)$ of $F$ with a tubular neighborhood of zero-section in $NF$. We may assume $U(F)$ is $S_n$-invariant, and the map $U(F) \to F$ is $S_n$-equivariant (since all data and operations are $S_n$-equivariant). The map $\Pi$ induces a natural splitting (upon shrinking $U(F)$) $U(F) \cong F \times U_0$, where $U_0$ is a neighborhood of $0$. Thus, we may lift vector fields $\d_{t_i}$ on $U_0$ to commuting vector fields $X_i$ on $U(F)$. Then $X_i$ varies $\rho_i$ while keeping all other $\rho_j$ fixed. It is easy to see $X_i$ are cm vector field. Let $\wt I_i: T^*\R^n \to \R$ be given by $\wt I_i(x,p) = p \cdot X_i(x)$, where $(x,p) \in T^*\R^n$. Then $\{\wt I_i, \wt I_j\}=0$ with standard symplectic structure. Since $\wt I_i$ are smooth cm functions on $\C^n$, they descend to Poisson commuting smooth function on $\Sym^n \C$. 
\end{proof}

Recall that Definition \ref{def: box opens}, Lemma \ref{box opens cover}, and Propositions \ref{box opens have cm boundary} and \ref{prop: transverse boundaries} above construct a cover of $\Sym^n \R$ by sets $U^{\ba,\delta}_{i,j}$ with $i+j=n$, whose boundaries are cut out by transverse center of mass functions.  Per Lemmas \ref{lm:sectorial-base-condition} and \ref{lem: sectorial collection criterion}, we may pull them back by $\Re$ to get a sectorial cover of $\Sym^n \C$.\footnote{Let us note that our construction of covers of $\Sym^n\C$ by pulling back from $\Sym^n \R$ is partly inspired by  \cite{Kapranov-Schectman-shuffle}.} We will need the following compatibility:

\begin{lemma} \label{lem: compatibility}
    Consider some $-2 < a < b < -1 < 1 < c < d < 2$.  Write $n = i + j + k + l$ and $m=j+k$. Fix $\ba=(-1<a_0<\cdots < a_n<1)$, and let $\ba' = (-1 < a_i < \cdots a_{i+m} < 1)$ be a subset of $\ba$. Under the identification 
    $$U^{\ba,\delta}_{i+j, k+l} \cap (\Sym^{i}(-\infty, a) \times \Sym^{j+k}(b, c) \times  \Sym^{k}(d, \infty)) = \Sym^{i}(-\infty, a) \times  (U^{\ba',\delta}_{j,k} \cap \Sym^{j+k}(b,c)) \times \Sym^{k}(d, \infty),$$ 
    the defining function $\rho^{\ba, \delta}_{k,n-k}$ for $U^{\ba,\delta}_{i+j, k+l}$ , when restricted to this intersection, is the pullback of the corresponding defining function for $ U^{\ba', \delta}_{j,k}$. Consequently, the witness function $I$ constructed in Lemma \ref{lm:sectorial-base-condition} for $\Re^{-1} U^{\ba,\delta}_{i+j, k+l}$ (from the defining function $\rho$) is the  pullback of the corresponding witness function for $\Re^{-1} U^{\ba',\delta}_{j,k}$. 
\end{lemma}
\begin{proof}
The boundary of the (unsmoothed) $U^{\ba}_{i+j,k+l}$ is given by certain boxes' boundary hitting certain points $a_{i+j}$. After intersecting with $(\Sym^{i}(-\infty, a) \times \Sym^{j+k}(b, c) \times  \Sym^{k}(d, \infty))$, the only change is the possible size of the boxes changed, hence the defining equations of the boundary of $U^{\ba}_{i+j,k+l}$ and $U^{\ba'}_{i+j,k+l}$ agree on the nose. The subsequent smoothed version follows as well. The claim of witness function also follows since the metric for taking gradient of $\rho$ are compatible. 
\end{proof}

\begin{lemma} \label{the cover}
    Fix a gluing configuration, and 
    consider the projection $\pi_n: \Sym^n \Sigma \to \Sym^n([-2,2]) \subset \Sym^n \R$.  
    Fix a Liouville form on $\Sym^n \Sigma$, adapted to the gluing in the sense of Def. \ref{def: adapted}. 
    Let $U^\ba_{i,j}$ and $U_{i,j}^{\ba,\delta} \subset \Sym^n(\R)$ be the open sets of Def. \ref{def: box opens}.  
    Then $(\pi_n)^{-1}(U^\ba_{i,j})$ are sectorial, and $(\pi_n)^{-1}(U_{i,j}^{\ba,\delta})$ form a sectorial cover.
\end{lemma}
\begin{proof}
    We must construct the witness functions and check their properties.  Let us pass to the intersection with an open set $\Omega_{l,k,r}$ from the cover \eqref{edge versus complement cover}.  Then the boundary of 
    $(\pi_n)^{-1}(U^\ba_{i,j})$ is pulled back from
    the boundary of $\mathrm{Re}^{-1}(U_{i-r, j-l}) \subset \Sym^{i+j-r-l}(b,c)$.   
    
    Because the Liouville form on $\Sym^n \Sigma$ is adapted,  if we pull back the witness function $I$ from $\mathrm{Re}^{-1}(U^{\ba'}_{i-r, j-l})$ to $\Omega_{l,k,r}$, it has the properties of a sectorial witness for $(\pi_n)^{-1}(U_{i,j})$ over $\Omega_{l,k,r}$, by Lemma \ref{lm:sectorial-base-condition}. 
    Now Lemma \ref{lem: compatibility} ensures that witness functions for different $l,k,r$ agree on overlaps and hence define a global function.  
    
    The same holds for the $U_{i,j}^{\ba,\delta}$.  For these we can check additionally that their boundaries meet sectorially   by Prop. \ref{prop: transverse boundaries} and Lemma \ref{lem: sectorial collection criterion}.
\end{proof}

\subsection{Weinstein-ness}

\begin{proposition} \label{weinstein-ness}
Let $\Sigma = \Sigma_L \cup_A \Sigma_R$ be a gluing configuration, $\pi: \Sigma \to [-2,2]$, $\pi_n = \Sym^n \pi: \Sym^n \Sigma \to \Sym^n [-2,2] \In \Sym^n \R$. Let $\varphi_{n,t}$ be the smoothed K\"ahler potential on $\Sym^n \Sigma$. Then the Liouville sectors $\Sigma_{k,n-k}:=\pi_n^{-1} \overline{U_{k,n-k}^{\ba, \delta}}$ are Weinstein sectors. 
\end{proposition}

\begin{proof}
The function $\varphi_{n,t}$ restricted to $\Sigma_{k,n-k}$ is a proper and smooth strictly psh function. The induced Liouville flow $Z_{n,t}$ preserves the boundary $\d \Sigma_{k,n-k}$. Hence, we only need to check that there is a hypersurface $F \In \d \Sigma_{k,n-k}$ such that, (1) $F$ is transverse to the characteristic foliation, (2) $F$ is preserved by the Liouville flow $Z_{n,t}$, and (3) $(F, \lambda_{n,t}|_F, \varphi_{n,t}|_F)$ is a Weinstein hypersurface. 

Let $U=U_{k,n-k}^{\ba, \delta} \In \Sym^n [-2,2]$ with defining function $\rho = \rho_{k,n-k}^{\ba, \delta}: \Sym^n[-2,2] \to \R$. Consider their  lifts $\wt U = \wt {U^{\ba, \delta}_{k,n-k}} \In [-2,2]^n$ and $\wt \rho = \wt {\rho^{\ba, \delta}_{k,n-k}}: [-2,2]^n \to \R$. Similarly, let $\wt \pi_n: \Sigma^n \to [-2,2]^n$, and $\wt \Sigma_{k,n-k}:= \wt \pi_n^{-1}(\wt U)$. 

Over the boundary $\d \wt U$, we have the short exact sequence of bundles
$$ 0 \to N^*_{\d \wt U} (-2,2)^n \to T^*(-2,2)^n|_{\d \wt U} \to T^* (\d \wt U) \to 0 . $$
Using the Euclidean metric on $T^*(-2,2)^n$, we have a orthogonal splitting 
$$ T^*(-2,2)^n|_{\d \wt U} = N^*_{\d \wt U} (-2,2)^n \oplus T^* (\d \wt U)$$

The subset $A^o = \sqcup_{r=1}^d T^*(-2,2) \In \Sigma$ has a canonical $d$-sheeted covering map $A^o \to T^*(-2,2)$, hence restricts to a $d^n$-sheeted map $\pi^n_A: (A^o)^n \to T^*(-2,2)^n$. We define 
$$ \wt F^o_{k,n-k} = (\pi^n_A)^{-1}(T^* (\d \wt U)) \In (A^o)^n \cap \d \wt \Sigma_{k,n-k}. $$

More generally, for cut $-2 < a < b < -1 < 1 < c < d < 2$, with $b-a > 2r_n, d-c>2r_n$, and $n=l+m+r$, we may similarly define $\wt F_{l, k, r} \In \wt \Omega_{l,m,r} \cap \d \wt \Sigma_{k,n-k}$ by pulling back from $\wt F^o_{k-l, n-k-r} \In (A^o)^m \cap \wt \Sigma_{k-l, n-k-r}$.

These $\wt F_{l,k,r}$ are compatible by Lemma \ref{lem: compatibility}, hence patch together to give $\wt F \In \d \wt \Sigma_{k, n-k}$. 

We claim that
\begin{enumerate}
    \item For product exact symplectic structure on $\Sigma^n$, induced by the direct sum K\"ahler potential $\phi_n$, $\wt F$ is transverse to the Hamiltonian flow $X_{(\wt \pi_n)^*\wt \rho}$ on $\d \wt \Sigma_{k,n-k}$,  preserved under Liouville flow $\wt Z_{\phi_n}$, and contains all the zero loci of $\wt Z_{\phi_n}$.  
    \item $\wt F$ descends to a smooth submanifold $F$ in $\d \Sigma_{k,n-k} \In \Sym^n \Sigma$. 
    \item $F$ is transverse to the Hamiltonian flow $X_{( \pi_n)^* \rho}$ on $\d  \Sigma_{k,n-k}$,  preserved under Liouville flow $ Z_{\varphi_{n,t}}$, and contains all the zero loci of $Z_{\varphi_n}$. 
\end{enumerate}
To see the first point, we may work in local chart $\wt \Omega_{l,m,r}$, where $\wt F_{l,m,r} = \wt F \cap \wt \Omega_{l,m,r}$ is cut out by some functions in $T^*(b,c)^m$. This then reduces to the standard cotangent bundle case, which is easy to check. 

To see the second point,  we note that $\wt F$ is cut out by equations that only depends on center of mass factors (Eq. \eqref{re inverse product}), hence descend to a smooth manifold under the quotient. 

To see the third point, we note that all the Hamiltonian flow is happening on the center of mass factor; the Liouville flow $Z_{\varphi_{n,t}} = Z_{cm} + Z_{int}$ has the standard form $Z_{cm}$ part, where as the defining equation of $F$ only has the cm part, hence $Z_{\varphi_{n,t}}$ preserves $F$. And the claim about $F$ containing all zero-loci follows from the same statement pre-quotient. 
\end{proof}

\subsection{Proof of Theorems \ref{thm: functor} and \ref{thm: intro cover}}

By now we have established all the assertions in Theorems \ref{thm: functor} and \ref{thm: intro cover}; let us collect them here. 
  
\begin{proof}[Proof of Theorem \ref{thm: functor}]
    Suppose given a map $\Sigma' \to \Sigma$ in $\mathcal{O}$, where $\Sigma$ has no boundary.  Then there is a gluing configuration given by slightly thickening $\Sigma'$ and  $\Sigma'' := \Sigma \setminus \Sigma'$.  We define $\Sym^{(n)}(\Sigma')$ to be the $(\pi_n)^{-1} U_{0,n}$ of Lemma \ref{the cover}; said lemma establishes sectoriality.  We checked Weinstein-ness in Proposition \ref{weinstein-ness}. It is evident from the construction that this $\Sym^{(n)}(\Sigma')$ in fact depends only on $\Sigma'$ (and not on $\Sigma$), and by construction it comes with an embedding into $\Sym^{(n)}(\Sigma)$.  
\end{proof}

\begin{proof}[Proof of Theorem \ref{thm: intro cover}]
    Consider the cover by the sectors-with-corners
    $\pi_n^{-1}(U_{n_L,n_R})$ of Lemma \ref{the cover}. 
    Per Lemma \ref{U intersections}, the intersections are  $\pi_n^{-1}(U_{n_L,m_1, \ldots, m_k,n_R})$, which are sectors-with-corners by the argument of Lemma \ref{the cover}.  Lemmas \ref{U inclusions}, \ref{U intersections}, and \ref{U product} show  the \v{C}ech diagram of this cover by the  has the structure of the Bar diagram; indeed, there's a deformation to $\Sym(\mathrm{Bar}(\Sigma_L|A|\Sigma_R))$ given by introducing $\delta$ in the factors of the product of \ref{U product}.  On the other hand, the \v{C}ech diagram of the cover $\pi_n^{-1}(U_{n_L,n_R})$ also obviously has a deformation to the \v{C}ech diagram of the cover by $\pi_n^{-1}(U_{n_L,n_R}^\delta)$, which is sectorial per Lemma \ref{the cover}.\footnote{This messing around with smoothings is the price we pay for not having defined ``sectorial cover by sectors-with-corners''  and directly checked that $\pi_n^{-1}(U_{n_L,n_R})$ is such a cover.}
\end{proof}

\section{Calculations}

\subsection{Three easy pieces}

We have the following basic calculations:    

\begin{example}
    Let $D$ be the half disk.  Then
     \begin{eqnarray*}
        \Fuk(\Sym^0(D)) & = & \Fuk(\bullet) = k-\mathrm{mod} \\
        \Fuk(\Sym^{\ge 1}(D)) & = & 0.
    \end{eqnarray*}
\end{example}

\begin{example}
    Let $I = T^*[-1,1]$.  Then:    
    \begin{eqnarray*}
        \Fuk(\Sym^0(I)) & = & \Fuk(\bullet) = k-\mathrm{mod} \\
        \Fuk(\Sym^1(I)) & = & \Fuk(I) = k-\mathrm{mod} \\
        \Fuk(\Sym^{\ge 2}(I)) & = & 0.
    \end{eqnarray*}
\end{example}

\begin{example}
    Let $T$ be the open pair of pants / disk with three stops / trinion.  Then:    
    \begin{eqnarray*}
        \Fuk(\Sym^0(T)) & = & \Fuk(\bullet) = k-\mathrm{mod} \\
        \Fuk(\Sym^1(T)) & = & \Fuk(T) = \{\mbox{exact triangles}\} \\
         \Fuk(\Sym^2(T)) & = & k-\mathrm{mod} \\
        \Fuk(\Sym^{\ge 3}(T)) & = & 0.
    \end{eqnarray*}
    There are three commuting actions of $\Fuk(\Sym(I))$, one for each stop.  The nontrivial data of such an action are  maps 
    $\Fuk(\Sym^0(T)) \xrightarrow{\delta^0_{[1]}} \Fuk(\Sym^1(T))\xrightarrow{\delta^1_{[1]}} \Fuk(\Sym^2(T))$.  These are given as follows on one of the legs: 
    \begin{eqnarray*}
        V \xrightarrow{\delta^0_{[1]}} & [V \to V \to 0 \xrightarrow{[1]}] & \\ 
        & [X \to Y \to Z \xrightarrow{[1]}] & \xrightarrow{\delta^1_{[1]}}  Z
    \end{eqnarray*}
    and by the cyclic permutations of this on the other legs:
        \begin{eqnarray*}
        V \xrightarrow{\delta^0_{[2]}} & [0 \to V \to V \xrightarrow{[1]}] & \\ 
        & [X \to Y \to Z \xrightarrow{[1]}] & \xrightarrow{\delta^1_{[2]}}  X
    \end{eqnarray*}
    and
        \begin{eqnarray*}
        V \xrightarrow{\delta^0_{[3]}} & [V[1] \to 0 \to V \xrightarrow{[1]}] & \\ 
        & [X \to Y \to Z \xrightarrow{[1]}] & \xrightarrow{\delta^1_{[3]}}  Y.
    \end{eqnarray*}
    The whole setup has a curious twisted self-adjointness:     \begin{equation}\label{werd} (\delta_{[1]}^0)^R = \delta_{[2]}^1, \qquad (\delta_{[2]}^0)^R = \delta_{[3]}^1, \qquad (\delta_{[3]}^0)^R = \delta_{[1]}^1.
    \end{equation}
\end{example}

The only assertions above which go beyond the standard and elementary calculation of Fukaya categories of surfaces are the vanishing, the determination of $\Fuk(\Sym^2(T))$, and the description of the maps $\Fuk(\Sym^1(T)) = \Fuk(T) \to \Fuk(\Sym^2(T))$.   These are well known in other formulations for Fukaya categories of symmetric powers of surfaces with boundary.
Let us check them in ours. 

Consider the covering map $w=z^n : \C_z \to \C_w$. Equip $\C_{5 < \Re w < 10}$ with K\"ahler potential $\varphi(w) = \Im(w)^2$ and pullback to $\varphi_z$ on $\C_z$. Extend 
    $\varphi_z$ to smooth strictly psh on $\C_z$. \footnote{This is doable, for example, set
    $$ \varphi(z) = Im(z^n)^2 + \epsilon \chi(z/\delta) |z|^2 $$
    where $\chi(z)$ is a smooth bump function locally constant $1$ near $z=0$ and vanishes outside $|z|>1$. The first term is strictly psh away from $z=0$, and the second term provides positivity at $z=0$.}
    We write  
    $$D_n: = \{ \Re(z^n) \leq 10\}$$ for the resulting  `disk with $n$ stops'. 
    
\begin{proposition}
    $\Fuk (\Sym^{\ge n}(D_n))=0$. 
\end{proposition}
\begin{proof}
    We show all Lagrangians are displaceable by writing an appropriate Hamiltonian. 
    Define 
    $$ W_{n,k} (z_1, \cdots, z_k) = \sum_{i=1}^k z^n. $$
    Then it is clear that $W_k$ has no critical points in $\Sym^k \C$, since $W_k$ is one coordinate function of $\Sym^k \C$.
    Consider the Hamiltonian flow $H = \Im(W_k)$, then $X_{\Im(W_{n,k})} = -\nabla_{\Re(W_{n,k})}$. Hence Hamiltonian flow will decrease $\Re(W_{n,k})$, all the way to $-\infty$, no critical point. The flow is defined on the subset $\Sym^k(D_n)$ (or its sectorial smoothing) as well.\footnote{It is presumably true, but we have not proven, that our $\Sym^{k}(D_n)$ is in fact the sector associated to $W_{n,k}$.} 
\end{proof}

We also compute $\Fuk(\Sym^2(T=D_3))$ directly. Use coordinates $A=z_1+z_2, B=z_1 z_2$ on $\Sym^2 \C$, then the superpotential $W=W_{3,2}:\Sym^2(\C) = \C^2_{A,B} \to \C$
$$ W=z_1^3 + z_2^3 =  A^3 - 3 AB. $$ 
have one Morse critical point at $A=0, B=0$. We may again use $\nabla_{\Re(W)}$ to show that all Lagrangian which do not intersect the thimble can be displaced. We conclude $\Fuk(\Sym^2(D_3)) \cong k-\mathrm{mod}$.  Consider now the image of the inclusion of two strips on different legs of $D_3$, inducing $T^*I \times T^*I \to \Sym^2(D_3)$.  One can see geometrically that this map is forward stopped, so the induced map on Fukaya categories is fully faithful, hence sends the product of cotangent fibers to the generator of $k-\mathrm{mod}$.  This, together with the fact the vanishing $\Fuk(\Sym^2(I)) = 0$ implies $\delta^2 =0$, determines all asserted structure maps for $\Fuk(\Sym(T))$

\subsection{Cutting corners}

To simplify formulas going forward, we will write $\mathbf{k}:= k-\mathrm{mod}$, $\mathbf{k}[\epsilon] := \Fuk\Sym(I)$,  $\mathbf{T}$ for the category of exact triangles, and $\mathbb{T}:= \Fuk\Sym(T)$.

Here we will prove a slightly more general version of Theorem \ref{thm:cdw}:

\begin{theorem} \label{thm:tensor to pushout}
Let $M$ be a $\Z$-graded $\mathbf{k}[\epsilon] \otimes \mathbf{k}[\epsilon]$-module.  Then there is a canonical isomorphism
$$\left(\mathbb{T} \bigotimes_{k[\epsilon] \otimes k[\epsilon]} M\right)^n \xleftarrow{\sim} ( \mathbf{T} \otimes M^{n-1}) \!\!\!\!\! \coprod_{M^{n-1} \oplus M^{n-1}} \!\!\!\!\! M^n$$
Here, the maps $M^{n-1} \oplus M^{n-1} \to M^{n-1} \otimes \mathbf{T}$ are given by tensoring up  $\delta^0_{[2]} \oplus \delta^0_{[1]}: \mathbf{k} \oplus \mathbf{k} \to \mathbf{T}$, and the maps $M^{n-1} \oplus M^{n-1} \to M^n$ are given by the two $\epsilon$ actions.  Likewise, the $k[\epsilon]$ actions on $\mathbb{T}$ are by $\delta_{[2]}$ and $\delta_{[1]}$. 
\end{theorem}
\begin{remark}
    The relation to Theorem \ref{thm:cdw} as stated is the following.  Given two $\Z$-graded $\mathbf{k}[\epsilon]$-modules, $V, W$, we have (tautologically from the definition plus the formulas in the previous subsection) $$\mathbb{T} \otimes_{\mathbf{k}[\epsilon] \otimes \mathbf{k}[\epsilon]} (V \otimes W) = V \star_T W.$$ Meanwhile the expression on the right hand side of Theorem \ref{thm:tensor to pushout} tautologically specializes to one of the formulas given in \cite[p. 22, p. 26]{Christ-Dyckerhoff-Walde}  for the `coproduct totalization' of the bicomplex $V \otimes W$. (The authors of loc. cit. prefer to describe $\mathbf{T}$ as the representations of the $A_2$ quiver $\bullet \to \bullet$.) 
    The added generality of Theorem \ref{thm:tensor to pushout} over Theorem \ref{thm:cdw} is just that the above formulation does not require $M$ to be of the form $V \otimes W$, or for that matter, to have a bigrading.\footnote{In fact, the argument goes through without $M$ having any grading at all: note in the proof, we use the grading on the $X$ and $A$ factors, but not on the $Y=M$ factor.} This is necessary in order to glue the $T$ to not-disconnected surfaces with two chosen boundaries.
\end{remark}
\begin{remark}
    Coherent associativity of the totalization was not previously established 
    \cite[Remark 4.4.4]{Christ-Dyckerhoff-Walde}.  After Theorem \ref{thm:tensor to pushout}, we can  import it from  the coherent associativity of $\star_T$, which, as noted above, follows immediately from Theorem \ref{thm: gluing}.    
\end{remark}
\begin{remark} \label{rem: RM versus CDW} 
    When the relevant maps admit adjoints, we may  trade pushouts for pullbacks, which are typically easier to compute.  (This is a standard trick; see e.g. \cite{nadler2016wrapped, gammage-shende} for some uses in computing Fukaya categories, and a discussion at length in the context of totalization in \cite{Christ-Dyckerhoff-Walde}.)  For $\mathbf{k}[\epsilon]$-modules where the maps  $\epsilon$ action ($M^{n-1} \to M^n$) admits a right adjoint -- as is always the case, for formal reasons,\footnote{Our Fukaya categories are by definition ind completions of the small categories which symplectic geometers typically call Fukaya categories.  The covariant functoriality  of partially wrapped Fukaya categories under sector inclusion is also originally defined on the small categories, so after ind completion it preserves compact objects, i.e. is a  left adjoint of a left adjoint.  So in fact we may pass to right adjoints while still remaining in $Pr^L$.  A not-formal question is whether or not these adjoints themselves admit right adjoints in $Pr^L$, or in other words, whether the adjoints would have existed on the small categories.} for modules  arising from Fukaya categories of surfaces -- the right hand side of Theorem \ref{thm:tensor to pushout} can be computed as:
    $$(\mathbf{T} \otimes M^{n-1}) \!\!\!\!\! \!\!\!  \bigtimes_{M^{n-1} \oplus M^{n-1}}  
    \!\!\!\!\!\!\!\!  M^n.$$ 
    The maps $M^n \to M^{n-1}$ are the right adjoints of the $\epsilon$ actions, and the maps $\mathbf{T} \otimes M^{n-1} \to M^{n-1} \oplus M^{n-1}$ are the right adjoints of $\delta^0_{[2]}$ and $\delta^0_{[1]}$.  Thus an object of this pushout is an object $m \in M^n$, plus a morphism $(\epsilon_2)^R(m) \to (\epsilon_1)^R(m)$.  
    In case $M = V \otimes W$ and $m = v \otimes w$, the $\epsilon$ each only act on one factor, so we are asking for a map $v \otimes \epsilon_2^R(w) \to \epsilon_1^R(v) \otimes w$.  Up to  conventions, this is   the  description of objects of $V \ootimes W$ given by Rouquier and Manion \cite[p.  6]{rouquier-manion}.  They give an action their categorified $\mathfrak{gl}(1|1)^+$ on the result, which one can see from the following proof agrees with the action induced from the third (unused in gluing) $k[\epsilon]$ action on $\mathbb{T}$.  Note that this passage to adjoints accounts for the fact that our $\epsilon$ is a raising operator (increases $n$ in $\Sym^n$), while theirs is a lowering operator.  
\end{remark}

\begin{proof}[Proof of Theorem \ref{thm:tensor to pushout}]
    Let us first consider a general expression of the form $X \otimes_A Y$, where 
    $A$ is a monoidal category and $X$ and $Y$ are, respectively, left and right $A$-modules.  Then 
    we have in general 
    $$  \colim \left( \cdots X \otimes A \otimes A \otimes Y   \,\, \substack{\rightarrow\\[-.6ex] \rightarrow \\[-.6ex] \rightarrow}   \,\, X \otimes A \otimes Y  \,\, \substack{\rightarrow\\[-.6ex] \rightarrow } \,\, X \otimes Y \right) = X \otimes_A Y.$$
    Now assume further that $A$ is $\N$-graded with $A^0 = \mathbf{k}$, $X$ is $\N$-graded, and $Y$ is $\Z$-graded.    Then, after removing redundant terms, 
    the above expression allows the degree $n$ part, $(X \otimes_A Y)^n$, to be computed as the colimit of a diagram of the following shape.   The vertices are the vertices of a barycentric subdivision of a $n$-simplex. To avoid confusion, we will refer to the vertices of the original simplex as {\em corner vertices}.  The arrows are oriented as follows: on each face, the barycenter is a source, and the corner vertices are sinks.  The barycenter of a $k$-face is labelled by a category of the form $X^{n_X} \otimes A^{m_1} \otimes \cdots \otimes A^{m_k} \otimes Y^{n_Y}$, where the superscripts are integers which sum to $n$, $n_X \ge 0$, and the $m_k > 0$;  we will indicate this category by the tuple $(n_X, m_1, m_2, \cdots, m_k, n_Y)$.  The arrows in the diagram are the monoidal structure maps, i.e. correspond to merging adjacent entries in the tuple.  The colimit is understood homotopically, i.e. over the category whose nerve is the solid simplex on which we described the diagram.  We will denote this diagram $\Delta^n(X|A|Y)$.

    We will use two basic facts about simplifying such diagrams.  If $v$ is a corner vertex, note that the sub-diagram on all vertices mapping to $v$ is precisely the star of $v$; we write $\mathrm{link}(v) := \mathrm{star}(v) \setminus v$.  There's a natural map from the colimit over $\mathrm{link}(v)$ to the value at $v$; if this map is an isomorphism, then $\Delta^n(X|A|Y)$ and $\Delta^n(X|A|Y) \setminus v$ have the same colimit of the diagram.    Second,  the colimit of $\Delta^n(X|A|Y) \setminus v$ is equivalent to the colimit on the facet of $\Delta^n(X|A|Y)$ opposing $v$. 

    Now we turn to the diagram of interest: $X = \mathbb{T}$, $A= k[\epsilon] \otimes k[\epsilon]$, and $Y = M$ arbitrary.  Observe that then $X^{\ge 3} = 0$ and $A^{\ge 3} = 0$.  Consider the corner vertex $(5,n-5)$.  The corresponding category is $\mathbb{T}^5 \otimes M^{n-5} = 0$.  The nearest neighbors are the points $(a, 5-a, n-5)$.  But every such category is zero, so indeed $(5,n-5)=0$ is the pushout of its link, and may be removed.  The same works a fortiori for $(k, n-k)$ for any $k\ge 5$, so we may iteratively collapse to the diagram on the 4-face of $\Delta^n(\mathbb{T}|\mathbf{k}[\epsilon] \otimes \mathbf{k}[\epsilon]|M)$ spanned by corner vertices $(4|n-4), \ldots, (0|n)$.  
    
    Consider the corner $(4|n-4)$.  The category there is $\mathbb{T}^4 \otimes M^{n-4}=0$.  The only non-zero immediately adjacent vertex is $(2|2|n-4)$ with value $\mathbb{T}^2  \otimes (\mathbf{k}[\epsilon] \otimes \mathbf{k}[\epsilon])^2 \otimes M^{n-2} = M^{n-2}$.  In the link of $(4|n-4)$, we have a subdiagram 
    $$0 = (3|1|n-2) \leftarrow  (2|1|1|n-2) \rightarrow  (2|2|n-2) = M^{n-2}.$$  The map 
    $$(2|1|1|n-2) = \mathbf{k} \otimes (\mathbf{k} \oplus \mathbf{k}) \otimes (\mathbf{k} \oplus \mathbf{k})  \otimes M^{n-2} \to M^{n-2}$$ is easily seen to be surjective (the $M^{n-2}$ just comes along for the ride, and the map $(2|1|1|0) \to (2|2|0)$ is obviously surjective).  It follows that the pushout of the link of $(4|n-4)$ is zero, and so we can again cut the corner and project onto the 3-simplex with corners $(3, n-3),\ldots, (0,n)$.

    Again the corner vertex vanishes: $(3|n-3) = 0$. The only potentially nonzero adjacent terms are $(2|1|n-3) = \mathbf{k} \otimes \mathbf{k}^{\oplus 2} \otimes M^{n-3}$ and $(1|2|n-3) = \mathbb{T} \otimes \mathbf{k} \otimes M^{n-3}$.  Again the $Y$ factors will come for the ride, so we suppress them (equivalently set $n=3$).   Now, 
    we have subdiagrams of the link 
    $$0 = (0|3| 0) \leftarrow (0|1|2|0) = \mathbf{k} \oplus \mathbf{k}  \twoheadrightarrow (1|2|0) = \mathbf{T} \qquad \qquad 0 = (0|3|0) \leftarrow (0|2|1|0) = \mathbf{k} \oplus \mathbf{k} \xrightarrow{\sim} (2|1|0) = \mathbf{k} \oplus \mathbf{k}. $$
    From these we see that the pushout of the link is again zero, and so we may cut the corner, and project to the 2-simplex on $(2| n-2), \ldots, (0|n)$. 

    Consider the corner $(2|n-2) = M^{n-2}$.  Its link is the diagram $(1|1|n-2) \leftarrow (0|1|1|n-2) \rightarrow (0|2|n-2)$ which we expand as
    $$\mathbf{T} \otimes (\mathbf{k} \oplus \mathbf{k}) \otimes M^{n-2} \leftarrow (\mathbf{k} \oplus \mathbf{k}) \otimes (\mathbf{k} \oplus \mathbf{k}) \otimes M^{n-2} \rightarrow \mathbf{k} \otimes  M^{n-2}$$
    We claim the pushout is in fact $M^{n-2}$.  Let's drop the $M^{n-2}$ to avoid cluttering notation, and decorate the $\mathbf{k}$ factors by $1$ and $2$ according as they come from the first or second copy of $\mathbf{k}[\epsilon]$:
    $$\mathbf{T} \otimes (\mathbf{k}_1 \oplus \mathbf{k}_2)  \leftarrow (\mathbf{k}_1 \oplus \mathbf{k}_2) \otimes (\mathbf{k}_1 \oplus \mathbf{k}_2) \rightarrow \mathbf{k} $$
    Now, the $\mathbf{k}_1 \otimes \mathbf{k}_1$ factor in the middle maps to $\mathbf{T} \otimes \mathbf{k}_1$ on the left, and dies on the right; similarly with $\mathbf{k}_2 \otimes \mathbf{k}_2$.  
    Thus the pushout of this diagram agrees with the pushout by the quotient of these pieces and their images; i.e. of 
    $$((\mathbf{T}/\mathbf{k}_1) \otimes \mathbf{k}_1) \oplus ((\mathbf{T} /\mathbf{k}_2)  \leftarrow 
    (\mathbf{k}_2 \otimes \mathbf{k}_1) \oplus (\mathbf{k}_1 \otimes \mathbf{k}_2)\rightarrow \mathbf{k}$$
    The left arrow is an isomorphism, so the colimit is $\mathbf{k}$, as desired.  We may cut the corner. 

    What remains is the diagram $(1|n-1) \leftarrow (0|1|n-1) \rightarrow (0|n)$, or, in other words
    $$\mathbf{T} \otimes M^{n-1} \leftarrow (\mathbf{k} \oplus \mathbf{k}) \otimes M^{n-1} \rightarrow M^n.$$
    This completes the proof.
\end{proof}
\newpage

\bibliographystyle{plain}
\bibliography{ref}

@article{perutz2008hamiltonian,
  title={Hamiltonian handleslides for Heegaard Floer homology},
  author={Perutz, Timothy},
  journal={arXiv preprint arXiv:0801.0564},
  year={2008}
}

@book{petersen2006riemannian,
  title={Riemannian geometry},
  author={Petersen, Peter},
  year={2006},
  publisher={Springer}
}

@article{richberg1967stetige,
  title={Stetige streng pseudokonvexe Funktionen},
  author={Richberg, Rolf},
  journal={Mathematische Annalen},
  volume={175},
  number={4},
  pages={257--286},
  year={1967},
  publisher={Springer}
}

@article{Ozsvath-Szabo-3manifolds,
  title={Holomorphic disks and topological invariants for closed three-manifolds},
  author={Ozsv{\'a}th, Peter and Szab{\'o}, Zolt{\'a}n},
  journal={Annals of Mathematics},
  pages={1027--1158},
  year={2004},
  publisher={JSTOR}
}

@article{Ozsvath-Szabo-knots,
  title={Holomorphic disks and knot invariants},
  author={Ozsv{\'a}th, Peter and Szab{\'o}, Zolt{\'a}n},
  journal={Advances in Mathematics},
  volume={186},
  number={1},
  pages={58--116},
  year={2004},
  publisher={Elsevier}
}

@book{Lipschitz-Ozsvath-Thurston,
  title={Bordered {H}eegaard {F}loer homology},
  author={Lipshitz, Robert and Ozsv{\'a}th, Peter and Thurston, Dylan},
  year={2018},
  publisher={American Mathematical Society}
}

@article{Ozsvath-Szabo-four-manifolds,
  title={Holomorphic triangles and invariants for smooth four-manifolds},
  author={Ozsv{\'a}th, Peter and Szab{\'o}, Zolt{\'a}n},
  journal={Advances in Mathematics},
  volume={202},
  number={2},
  pages={326--400},
  year={2006},
  publisher={Elsevier}
}

@article{Douglas-Manolescu,
  title={On the algebra of cornered {F}loer homology},
  author={Douglas, Christopher and Manolescu, Ciprian},
  journal={Journal of Topology},
  volume={7},
  number={1},
  pages={1--68},
  year={2014},
  publisher={Wiley Online Library}
}

@book{Douglas-Lipschitz-Manolescu,
  title={Cornered {H}eegaard {F}loer homology},
  author={Douglas, Christopher and Lipshitz, Robert and Manolescu, Ciprian},
  year={2019},
  publisher={American Mathematical Society}
}

@book{demailly1997complex,
  title={Complex analytic and differential geometry},
  author={Demailly, Jean-Pierre},
  year={1997},
  publisher={Universit{\'e} de Grenoble}
}

@article{rouquier-manion,
  title={Higher representations and cornered {H}eegaard {F}loer homology},
  author={Manion, Andrew and Rouquier, Raphael},
  journal={arXiv:2009.09627}
}

@article{varouchas1984stabilite,
  title={Stabilit{\'e} de la classe des vari{\'e}t{\'e}s k{\"a}hl{\'e}riennes par certains morphismes propres},
  author={Varouchas, Jean},
  journal={Inventiones mathematicae},
  volume={77},
  number={1},
  pages={117--127},
  year={1984},
  publisher={Springer}
}

@article{manolescu1,
  title={Nilpotent slices, {H}ilbert schemes, and the {J}ones polynomial},
  author={Manolescu, Ciprian},
  journal={Duke Math. J.},
  volume={131},
  number={1},
  pages={311--369},
  year={2006}
}

@article{manolescu2,
  title={Link homology theories from symplectic geometry},
  author={Manolescu, Ciprian},
  journal={Advances in Mathematics},
  volume={211},
  number={1},
  pages={363--416},
  year={2007},
  publisher={Elsevier}
}

@article{ALR,
    author = "Aganagic, Mina and LePage, Elise and Rapcak, Miroslav",
    title = "{Homological Link Invariants from Floer Theory}",
    journal = "arXiv:2305.13480"
}

@article{abouzaid-smith:khovanov,
  title={Khovanov homology from {F}loer cohomology},
  author={Abouzaid, Mohammed and Smith, Ivan},
  journal={Journal of the American Mathematical Society},
  volume={32},
  number={1},
  pages={1--79},
  year={2019}
}

@article{abouzaid-smith-arc,
  title={The symplectic arc algebra is formal},
  author={Abouzaid, Mohammed and Smith, Ivan},
  journal={Duke Mathematical Journal},
  volume={165},
  number={6},
  pages={985--1060},
  year={2016},
  publisher={Duke University Press}
}

@article{GPS1,
  title={Covariantly functorial wrapped {F}loer theory on {L}iouville sectors},
  author={Ganatra, Sheel and Pardon, John and Shende, Vivek},
  journal={Publications math{\'e}matiques de l'IH{\'E}S},
  pages={1--128},
  year={2019},
  publisher={Springer}
}

@article{GPS2,
  title={Sectorial descent for wrapped {F}ukaya categories},
  author={Ganatra, Sheel and Pardon, John and Shende, Vivek},
  journal={Journal of the American Mathematical Society},
  volume={37},
  number={2},
  pages={499--635},
  year={2024}
}

@article{GPS3,
  title={Microlocal {M}orse theory of wrapped {F}ukaya categories},
  author={Ganatra, Sheel and Pardon, John and Shende, Vivek},
  journal={Annals of Mathematics},
  volume={199},
  number={3},
  pages={943--1042},
  year={2024},
  publisher={Department of Mathematics, Princeton University Princeton, New Jersey, USA}
}

@book{cieliebak-eliashberg,
  title={From Stein to Weinstein and back: symplectic geometry of affine complex manifolds},
  author={Cieliebak, Kai and Eliashberg, Yakov},
  year={2012},
  publisher={American Mathematical Soc.}
}

@article{nadler2016wrapped,
  title={Wrapped microlocal sheaves on pairs of pants},
  author={Nadler, David},
  journal={arXiv:1604.00114}
}

@article{shende-microlocal,
  title={Microlocal category for {W}einstein manifolds via the h-principle},
  author={Shende, Vivek},
  journal={Publications of the Research Institute for Mathematical Sciences},
  volume={57},
  number={3},
  pages={1041--1048},
  year={2021}
}

@article{aganagic-knot-2,
  title={Knot categorification from mirror symmetry, part {II}: Lagrangians},
  author={Aganagic, Mina},
  journal={arXiv:2105.06039}
}

@article{aganagic-icm,
  title={Homological Knot Invariants from Mirror Symmetry},
  author={Aganagic, Mina},
  journal={arXiv:2207.14104}
}

@article{Honda-Tian-Yuan,
  title={Higher-dimensional {H}eegaard {F}loer homology and {H}ecke algebras},
  author={Honda, Ko and Tian, Yin and Yuan, Tianyu},
  journal={arXiv:2202.05593}
}

@article{Colin-Honda-Tian,
  title={Applications of higher-dimensional {H}eegaard {F}loer homology to contact topology},
  author={Colin, Vincent and Honda, Ko and Tian, Yin},
  journal={arXiv:2006.05701}
}

@article{Yuan-link,
  title={A link invariant from higher-dimensional {H}eegaard {F}loer homology},
  author={Yuan, Tianyu},
  journal={arXiv:2309.13241}
}

@article{Seidel-Smith,
  title={A link invariant from the symplectic geometry of nilpotent slices},
  author={Seidel, Paul and Smith, Ivan},
  journal={Duke Mathematical Journal},
  volume={134},
  number={3},
  pages={453--514},
  year={2006},
  publisher={Duke University Press}
}

@incollection{Johnson-Freyd,
  title={Heisenberg-picture quantum field theory},
  author={Johnson-Freyd, Theo},
  booktitle={Representation Theory, Mathematical Physics, and Integrable Systems: In Honor of Nicolai Reshetikhin},
  pages={371--409},
  year={2021},
  publisher={Springer}
}

@article{ADLSZ,
  title={Quiver {H}ecke algebras from {F}loer homology in {C}oulomb branches},
  author={Aganagic, Mina and Danilenko, Ivan and Li, Yixuan and Shende, Vivek and Zhou, Peng},
  journal={arXiv:2406.04258}
}

@article{Dyckerhoff-Jasso-Lekili,
  title={The symplectic geometry of higher {A}uslander algebras: symmetric products of disks},
  author={Dyckerhoff, Tobias and Jasso, Gustavo and Lekili, Yank$\iota$},
  journal={Forum of Mathematics, Sigma},
  volume={9},
  pages={e10},
  year={2021},
  organization={Cambridge University Press}
}

@article{Lazarev-injective,
  title={Simplifying {W}einstein {M}orse functions},
  author={Lazarev, Oleg},
  journal={Geometry \& Topology},
  volume={24},
  number={5},
  pages={2603--2646},
  year={2020},
  publisher={Mathematical Sciences Publishers}
}

@article{Lazarev-Sylvan-Tanaka,
  title={The infinity-category of stabilized {L}iouville sectors},
  author={Lazarev, Oleg and Sylvan, Zachary and Tanaka, Hiro},
  journal={arXiv:2110.11754}
}

@article{Auroux-bordered,
  title={Fukaya categories of symmetric products and bordered {H}eegaard-{F}loer homology},
  author={Auroux, Denis},
  journal={Journal of G{\"o}kova Geometry Topology},
  volume={4},
  pages={1--54},
  year={2010}
}

@article{Christ-Dyckerhoff-Walde,
  title={Complexes of stable $\infty$-categories},
  author={Christ, Merlin and Dyckerhoff, Tobias and Walde, Tashi},
  journal={arXiv:2301.02606}
}

@article{Rouquier-lecture,
    title={Higher tensor products},
    author={Rouquier, Rapha\"el},
    journal={\url{https://www.youtube.com/watch?v=FPJH5vvstPk}}
}

@book{lurie-topos,
  title={Higher topos theory},
  author={Lurie, Jacob},
  year={2009},
  publisher={Princeton University Press}
}

@book{lurie-algebra,
  title={Higher Algebra},
  author={Lurie, Jacob}, 
  publisher={\url{https://www.math.ias.edu/~lurie/papers/HA.pdf}}
}

@article{Dai-sectorial,
  title={Sectorial Decompositions of Symmetric Products of Surfaces},
  author={Dai, Xinle},
  journal={arXiv:2511.16584}
}

@article{Kapranov-Schectman-shuffle,
  title={Shuffle algebras and perverse sheaves},
  author={Kapranov, Mikhail and Schechtman, Vadim},
  journal={Pure and Applied Mathematics Quarterly},
  volume={16},
  number={3},
  pages={573--657},
  year={2020},
  publisher={International Press of Boston, Inc. Somerville, MA 02143, USA}
}

@article{lepage-shende,
  title={Aganagic's invariant is {K}hovanov homology},
  author={LePage, Elise and Shende, Vivek},
  journal={arXiv:2505.00327}
}

@article{nadler-shende,
  title={Sheaf quantization in {W}einstein symplectic manifolds},
  author={Nadler, David and Shende, Vivek},
  journal={arXiv:2007.10154}
}

@article{didedda,
  title={Symplectic higher {A}uslander correspondence for type {A}},
  author={Di Dedda, Ilaria},
  journal={arXiv:2311.16859}
}

@article{sarkar-wang,
  title={An algorithm for computing some {H}eegaard {F}loer homologies},
  author={Sarkar, Sucharit and Wang, Jiajun},
  journal={Annals of mathematics},
  pages={1213--1236},
  volume={171},
  number={2},
  year={2010},
  publisher={JSTOR}
}

@article{manolescu-ozsvath-sarkar,
  title={A combinatorial description of knot {F}loer homology},
  author={Manolescu, Ciprian and Ozsv{\'a}th, Peter and Sarkar, Sucharit},
  journal={Annals of Mathematics},
  pages={633--660},
  year={2009},
  publisher={JSTOR}
}

@article{khovanov-gl12,
  title={How to categorify one-half of quantum $\mathfrak{gl}(1| 2)$},
  author={Khovanov, Mikhail},
  journal={Banach Center Publications},
  volume={103},
  number={1},
  pages={211--232},
  year={2014}
}

@article{gammage-shende,
  title={Mirror symmetry for very affine hypersurfaces},
  author={Gammage, Benjamin and Shende, Vivek},
  journal={Acta Mathematica},
  volume={229},
  number={2},
  pages={287--346},
  year={2022},
  publisher={Lehigh University Bethlehem, Penn., USA}
}

@article{lurie-rotation,
  title={Rotation invariance in algebraic K-theory},
  journal={\url{https://www.math.ias.edu/~lurie/papers/Waldhaus.pdf}},
  author={Lurie, Jacob}
}

@article{CKNS,
  title={Perverse microsheaves},
  author={C{\^o}t{\'e}, Laurent and Kuo, Christopher and Nadler, David and Shende, Vivek},
  journal={arXiv:2209.12998}
}

@article{Gavela-Large-Ward,
  title={On arborealization, {M}aslov data, and lack thereof},
  author={Alvarez-Gavela, Daniel and Large, Tim and Ward, Abigail},
  journal={arXiv:2503.09783}
}

\end{document}